\documentclass[journal]{IEEEtran}

\usepackage{amsfonts}
\usepackage{amsthm}
\usepackage{graphicx}
\usepackage{subcaption}
\usepackage{lipsum}
\usepackage{multicol} \usepackage{mathtools}
\usepackage{multirow}
\usepackage{threeparttable}
\usepackage[table,pdftex,hyperref,svgnames]{xcolor} 
\usepackage{color, colortbl}
\definecolor{Gray}{gray}{0.9}

\usepackage[noend]{algorithmic}
\algsetup{indent=2em}

\usepackage{algorithm}

\usepackage{multicol} \usepackage{mathtools}
\usepackage{version,xspace}
\usepackage{url,doi}
\newcommand{\until}[1]{\{1,\dots, #1\}}

\newcommand{\subscr}[2]{#1_{\textup{#2}}}
 \newcommand{\setdef}[2]{\{#1
  \; | \; #2\}}

\newcommand{\map}[3]{#1: #2 \rightarrow #3}
\newcommand{\union}{\operatorname{\cup}}
\renewcommand{\dim}{\operatorname{dim}}
\newcommand{\intersection}{\ensuremath{\operatorname{\cap}}}

\DeclareMathOperator{\sinc}{sinc}

\newcommand\oprocendsymbol{\hbox{$\triangle$}}
\newcommand\oprocend{\relax\ifmmode\else\unskip\hfill\fi\oprocendsymbol}

\usepackage{enumitem} \renewcommand{\theenumi}{(\roman{enumi})}

\DeclareSymbolFont{bbold}{U}{bbold}{m}{n}
\DeclareSymbolFontAlphabet{\mathbbold}{bbold}
\newcommand{\vect}[1]{\mathbbold{#1}}

\newcommand{\scirc}{\raise1pt\hbox{$\,\scriptstyle\circ\,$}}
\newcommand{\real}{\mathbb{R}}

\newcommand{\torus}{\mathbb{T}} \newcommand{\N}{\mathbb{N}}

\newtheorem{theorem}{Theorem}
\newtheorem{lemma}[theorem]{Lemma}
\newtheorem{corollary}[theorem]{Corollary}
\newtheorem{remark}[theorem]{Remark}
\newtheorem{example}[theorem]{Example}
\newtheorem{definition}[theorem]{Definition}

\newcommand{\fK}{\subscr{f}{K}} \newcommand{\imagunit}{\mathrm{i}}
\newcommand{\bperp}{\mathrm{B}^{G}(\gamma)}
\newcommand{\prj}{\mathcal{P}}

\newcommand{\R}{\subscr{R}{eff}}

\DeclareMathOperator{\diag}{diag}
\DeclareMathOperator{\spn}{span}
\DeclareMathOperator{\Ker}{\mathrm{Ker}}
\DeclareMathOperator{\Img}{\mathrm{Img}}
\renewcommand{\top}{\mathsf{T}} 

\ifCLASSINFOpdf
\else
\fi
\title{Synchronization of Kuramoto Oscillators \\ via Cutset
  Projections}

\author{Saber~Jafarpour,~\IEEEmembership{Member,~IEEE,}
        and Francesco~Bullo,~\IEEEmembership{Fellow,~IEEE}

\IEEEcompsocitemizethanks{
  \IEEEcompsocthanksitem This work was supported in part by the
    U.S. Department of Energy (DOE) Solar Energy Technologies Office
    under Contract No. DE-EE0000-1583.
  \IEEEcompsocthanksitem Saber Jafarpour, and Francesco
  Bullo are with the Mechanical Engineering Department and the Center
  of Control, Dynamical-Systems and Computation, UC Santa Barbara, CA
  93106-5070, USA. {\tt
    \{saber.jafarpour,bullo\}@engineering.ucsb.edu}}}

\begin{document}

\maketitle

\begin{abstract}
  Synchronization in coupled oscillators networks is a remarkable
  phenomenon of relevance in numerous fields. For Kuramoto oscillators the
  loss of synchronization is determined by a trade-off between coupling
  strength and oscillator heterogeneity. Despite extensive prior work, the
  existing sufficient conditions for synchronization are either very
  conservative or heuristic and approximate. Using a
    novel cutset projection operator, we propose a new family of
  sufficient synchronization conditions; these conditions rigorously
  identify the correct functional form of the trade-off between coupling
  strength and oscillator heterogeneity. To overcome the need to solve a
  nonconvex optimization problem, we then provide two explicit bounding
  methods, thereby obtaining (i) the best-known sufficient condition
  for unweighted graphs based on the 2-norm, and (ii) the
  first-known generally-applicable sufficient condition based on the
  $\infty$-norm. We conclude with a comparative study of our novel $\infty$-norm
  condition for specific topologies and IEEE test cases; for most IEEE
  test cases our new sufficient condition is one to two orders of
  magnitude more accurate than previous rigorous tests.
\end{abstract}


\section{Introduction}

\paragraph{Problem description and literature review}

The phenomenon of collective synchronization appears in many different
disciplines including biology, physics, chemistry, and engineering. In
the last few decades, many fundamental contributions have been made in
providing and analyzing suitable mathematical models for
synchronizations of coupled oscillators~\cite{NW:1958,ATW:67}. Much
recent interest in studying synchronization has focused on systems
with finite number of oscillators coupled through a nontrivial
topology with arbitrary weights. Consider a system consists of $n$
oscillators, where the $i$th oscillator has a natural rotational
frequency $\omega_i$ and its dynamics is described using the phase
angle $\theta_i\in \mathbb{S}^1$. When there is no interaction between
oscillators, the dynamics of $i$th oscillator is governed by the
differential equation $\dot{\theta}_i=\omega_i$. One can model the
coupling between oscillators using a weighted undirected graph $G$,
where the interaction between oscillators $i$ and $j$ is proportional
to sin of the phase difference between angles $\theta_i$ and
$\theta_j$. This model, often referred to as the Kuramoto model, is
one of the most widely-used model for studying synchronization of
finite population of coupled oscillators. The Kuramoto model and its
generalizations appear in various applications including the study of
pacemaker cells in heart~\cite{DCM-EPM-JJ:87}, neural
oscillators~\cite{EB-JM-PH:04}, deep brain simulation~\cite{PAT:03},
spin glass models~\cite{GJ-JA-DB-ACCC-CPV:01}, oscillating
neutrinos~\cite{JP:98}, chemical oscillators~\cite{IZK-YZ-JLH:02},
multi-vehicle coordination~\cite{RS-DP-NEL:07},
synchronization of smart grids~\cite{FD-MC-FB:11v-pnas}, security
analysis of power flow equations~\cite{AA-SS-VP:81}, optimal
generation dispatch~\cite{JL-SHL:12}, and droop-controlled inverters
in microgrids~\cite{MCC-DMD-RA:93, JWSP-FD-FB:12u}.

Despite its apparent simplicity, the Kuramoto model gives rise to very
complex and fascinating behaviors~\cite{FD-FB:13b}. A fundamental
question about the synchronization of coupled-oscillators networks is
whether the network achieves synchronization for a given set of
natural frequencies, graph topology, and edge weights.  While various
notions of synchronization in Kuramoto models have been proposed,
phase synchronization and frequency synchronization are arguably the
most fundamental. A network of coupled oscillators is in phase
synchronization if all the oscillators achieve the same phase and it
is in frequency synchronization if all the oscillators achieve the
same frequency. While phase synchronization is only achievable for
uniform frequencies irrespective of the network
structure~\cite{AJ-NM-MB:04, RS-DP-NEL:07, PM-FP:05}, frequency
synchronization in Kuramoto oscillators is possible for arbitrary
frequencies, but depends heavily on the network topology and weights.

\paragraph{Prior sufficient or necessary conditions for frequency
  synchronization}

Frequency synchronization of Kuramoto oscillators has been studied
using various approaches in different scientific communities. In the
physics and dynamical systems communities, in the limit as number of
oscillators tends to infinity, the Kuramoto model is analyzed as a
first-order continuity equation~\cite{YK:84,GBE:85}. In
the control community, much interest has focused on the finite numbers
of oscillators and on connections with graph theory. The first
rigorous characterization of frequency synchronization is developed
for the complete unweighted
graphs~\cite{DA-JAR:04,REM-SHS:05,MV-OM:08}. The
works~\cite{DA-JAR:04,REM-SHS:05} present implicit algebraic equations
for the threshold of synchronization together with local stability
analysis of the synchronization manifolds. The same set of equations
has been reported in~\cite[Theorem 3]{MV-OM:08}, where a bisection
algorithm is proposed to compute the synchronization
threshold. Moreover, \cite[Theorem 4.5]{MV-OM:09} presents a
synchronization analysis for complete unweighted bipartite graphs. Via
nonsmooth Lyapunov function methods, \cite[Theorem 4.1]{FD-FB:10w}
characterizes the case of complete unweighted graphs with arbitrary
frequencies in a fixed compact support. For acyclic graphs a necessary
and sufficient condition for frequency synchronization is presented
in~\cite[Remark 10]{AJ-NM-MB:04} and~\cite[Theorem
  2]{FD-MC-FB:11v-pnas}. Inspired by this characterization for acyclic
graphs and using an auxiliary fixed-point equation, a sufficient
condition for synchronization of ring graphs is proved
in~\cite[Theorem 3, Condition 3]{FD-MC-FB:11v-pnas}.

Unfortunately, none of the techniques mentioned above can be extended
for characterizing frequency synchronization of Kuramoto model with
general topology and arbitrary weights. The early works~\cite[\S
  VII(A)]{AJ-NM-MB:04}~\cite[Theorem 2.1]{NC-MWS:08}
present necessary conditions for synchronization. As of today, the
sharpest known necessary conditions are given by~\cite{NA-SG:13} and
are associated to the cutsets of the graph. Beside these necessary
conditions, numerous different sufficient conditions have also been
derived in the literature. The intuition behind most of these
conditions is that the Kuramoto model achieves frequency
synchronization when the couplings between the oscillators dominate
the dissimilarities in the natural frequencies.
An ingenious approach based on graph theoretic ideas is proposed
in~\cite{AJ-NM-MB:04}: if $2$-norm of the natural frequencies of the
oscillators is bounded by some connectivity measure of the graph, then
the network achieves a locally stable frequency
synchronization~\cite[Theorem 2]{AJ-NM-MB:04}. Other $2$-norm
conditions have been derived in the literature using quadratic
Lyapunov function~\cite[Theorem 4.2]{NC-MWS:08}~\cite[Theorem
  4.4]{FD-FB:09z} and sinusoidal Lyapunov function~\cite[Proposition
  1]{AF-AC-WPL:10}. To the best of our knowledge, the tightest
$2$-norm sufficient condition for existence of stable synchronization
manifolds for general topologies is given
by~\cite[Theorem~4.7]{FD-FB:12i}. Moreover, using numerical simulation
on random graphs and IEEE test cases, it is shown that the necessary and
sufficient condition for synchronization of acyclic graph can be
considered as a good approximation for frequency synchronization of a
large class of graphs~\cite{FD-MC-FB:11v-pnas}. Despite all these deep
and fundamental works, up to date, the gap between the necessary and
sufficient conditions for frequency synchronization of Kuramoto model
is in general huge and the problem of finding accurate and provably
correct synchronization conditions is far from resolved. Finally, we
mention that, parallel to the above analytical results, a large body
of literature in synchronization is devoted to the numerical analysis
of synchronization for specific random graphs such as small-world and
scale free networks~\cite{MB-LMP:02,TN-AEM-YCL-FCH:03,YM-AFP:04}. We
refer the interested readers to~\cite{JAA-LLB-CJPV-FR-RS:05,FD-FB:13b,
  AA-ADG-JK-YM-CZ:08} for survey of available results on frequency
synchronization and region of attraction of the synchronized manifold
as well as to~\cite{GS-AP-UM-FA:12, YW-FJD:13,EM-RAF-AKT:16} for
examples of recent developments and engineering applications.

\paragraph{Contributions}
As preliminary contributions, first, for a given weighted undirected
graph $G$, we introduce the cutset projection matrix of $G$, as the
oblique projection onto the cutset space of $G$ parallel to the
weighted cycle space of $G$. We find a compact matrix form for the
cutset projection of $G$ in terms of incidence matrix and Laplacian
matrix of $G$ and study its properties, including its $\infty$-norm
for acyclic, unweighted complete graphs and unweighted ring graphs.
Secondly, for a given graph $G$ and angle $\gamma\in[0,\pi)$, we
  introduce the embedded cohesive subset $S^{G}(\gamma)$ on the
  $n$-torus.  This subset is larger than the arc subset, but smaller
  than the cohesive subset studied in~\cite{FD-FB:09z,FD-FB:13b}.  We
  present an explicit algorithm for checking whether an element of the
  $n$-torus is in $S^{G}(\gamma)$ or not. We show that, for a network
  of Kuramoto oscillators, achieving locally exponentially stable
  frequency synchronization and existence of a synchronization
  manifold are equivalent in the domain $S^{G}(\gamma)$, for every
  $\gamma\in [0,{\pi}/{2}]$.

Our main contribution is a new family of sufficient conditions for the
existence of synchronized solutions to a network of Kuramoto
oscillators.
We start by using the cutset projection operator to rewrite the
Kuramoto equilibrium equation in an equivalent edge balance form.
Our first and main set of sufficient conditions for synchronization is
obtained via a concise proof that exploits this edge balance form and
the Brouwer Fixed-Point Theorem. These conditions require the norm of
the edge flow quantity $B^{\top}L^{\dagger}\omega$ to be smaller than
a critical threshold; here $L$ is the graph Laplacian and $B$ is the
(oriented) incidence matrix.
This first main set of conditions have various advantages and one
disadvantage.
The first advantage is that the conditions apply to any graph
topology, edge weights, and natural frequencies.
The second advantage is that the conditions are stated with respect to
an arbitrary norm; in other words, one can select or design a
preferable norm to express the condition in.
Finally, our conditions bring clarity to a conjecture arising
in~\cite{FD-MC-FB:11v-pnas}: while focusing on separated connectivity
and heterogeneity measures results in overly-conservative estimates of
the synchronization threshold, using combined measures leads to
tighter estimates. Building on the work in~\cite{FD-MC-FB:11v-pnas},
our novel approach establishes the role of the combined connectivity
and heterogeneity measure $B^{\top}L^{\dagger}\omega$ and results in
sharper synchronization estimates.

The disadvantage of our first main set of conditions is that the
critical threshold is equal to the minimum amplification factor of a
scaled projection operator, that is, to the solution of a nonconvex
minimization problem.  Instead of focusing on this minimization
problem, we here contribute two explicit lower bounds on the critical
threshold and two corresponding sufficient conditions for
synchronization. First, when $p=2$ and the graph is unweightd, we present an explicit lower bound
on the critical threshold which leads to a sharper synchronization
test than the best previously-known $2$-norm test in the
literature. Second, we present a general lower bound on the critical
threshold which leads to a family of explicit $p$-norm tests for
synchronization. For $p\ne2$, these $p$-norm tests are the first
rigorous conditions of their kind. In particular, for $p=\infty$, the
$\infty$-norm test establishes rigorously a modified version of the
approximate test proposed in~\cite{FD-MC-FB:11v-pnas}. Specifically,
while the test proposed in~\cite{FD-MC-FB:11v-pnas} was already shown
to be inaccurate for certain counterexamples, our $\infty$-norm test
here is a correct, more-conservative, and generically-applicable
version of it.

One additional advantage of this work is that our unifying technical
approach is based on a single concise proof method, from which various
special cases are obtained as corollaries. In particular, we show that
our sufficient conditions are: equal to those in the literature for
acyclic graphs, sharper than those in the literature for unweighted
ring graphs, and slightly more conservative than those in the
literature for unweighted complete graphs. Finally, we apply our
$\infty$-norm test
to a class of IEEE test cases from the MATPOWER
package~\cite{RDZ-CEMS-RJT:11}. We measure a test accuracy as a
percentage of the numerically-computed exact threshold.  For IEEE test
cases with number of nodes in the approximate range
$100$\textemdash{}$2500$, we show how our test improves the accuracy
of the sufficient synchronization condition from $0.11\%\textendash{}
0.29 \%$ to $23.08 \% \textendash{} 43.70\%$.

\paragraph{Paper organization}
In Section~\ref{sec:kuramoto}, we review the Kuramoto model. In
Section~\ref{sec:perliminary}, we present some preliminary results,
including the cutset projection operator.
Section~\ref{sec:synchronization_conditions} contains this paper's
main results the new family of $p$-norm synchronization
tests. Finally, Section~\ref{sec:comparison} is devoted to a
comparative analysis of the new sufficient conditions.

\paragraph{Notation}\label{sec:notation}
For $n\in \N$, let $\vect{1}_n$ (resp. $\vect{0}_n$) denote the vector
in $\real^n$ with all entries equal to $1$ (resp. 0), and define the
vector subspace $\vect{1}_n^{\perp}=\setdef{x\in
  \real^n}{\vect{1}_n^{\top}x=0}$.  For $n\in \N$, the $n$-torus and
$n$-sphere are denoted by $\mathbb{T}^n$ and $\mathbb{S}^n$,
respectively. Given two points $\alpha,\beta\in \mathbb{S}^1$, the
clockwise arc-length between $\alpha$ and $\beta$ and the
counterclockwise arc-length between $\alpha$ and $\beta$ are denoted
by $\mathrm{dist}_{c} (\alpha,\beta)$ and $\mathrm{dist}_{cc}
(\alpha,\beta)$ respectively.  The geodesic distance between $\alpha$
and $\beta$ in $\mathbb{S}^1$ is defined by
\begin{equation}\label{def:geodesic-distance}
  |\alpha-\beta|=\min\{\mathrm{dist}_{c} (\alpha,\beta),
  \mathrm{dist}_{cc} (\alpha,\beta)\}.
\end{equation}
For $z\in \mathbb{C}$, the real and imaginary part of $z$ are denoted
by $\Re(z)$ and $\Im(z)$, respectively. For $x\in \real^n$ and $p\in
[1,\infty)$, the $p$-norm of $x$ is
  $\|x\|_p=\sqrt[\leftroot{-2}\uproot{2} p]{|x_1|^p+\ldots+|x_n|^p}$
  and the $\infty$-norm of $x$ is
  $\|x\|_{\infty}=\max\setdef{|x_i|}{i\in \{1,\ldots,n\}}$.  For $A\in
  \real^{n\times m}$, the $p$-norm of $A$ is
  $\|A\|_p=\max\setdef{\|Ax\|_p}{\|x\|_p=1}$.  We let $A^{\top}$ denote
  the transpose of $A$.  The eigenvalues of a symmetric matrix $A\in\real^{n\times
    n}$ are real and denoted by $\lambda_1(A)\le \ldots\le
  \lambda_n(A)$. For symmetric matrices $A,B\in \real^{n\times n}$, we
  write $A\preceq B$ if $B-A$ is positive semidefinite. The
  Moore\textendash{}Penrose pseudoinverse of $A\in \real^{n\times m}$
  is the unique $A^{\dagger}\in \real^{m\times n}$ satisfying
  $AA^{\dagger}A=A$, $A^{\dagger}AA^{\dagger}=A^{\dagger}$,
  $\left(AA^{\dagger}\right)^{\top}=AA^{\dagger}$, and
  $\left(A^{\dagger}A\right)^{\top}=A^{\dagger}A$. Given a set
  $S\subseteq \real^m$ and matrix
  $A\in \real^{n\times m}$, we define the set $AS\subseteq \real^{n}$
  by $ AS =\{A\mathbf{v}\mid
  \mathbf{v}\in S\}$. Given subspaces $S$ and
  $T$ of $\real^n$, the \emph{minimal angle} between $S$ and $T$ is
  $\arccos(\max\setdef{x^{\top}y}{x\in S,\ y\in T,
    \|x\|_2=\|y\|_2=1})$.

If the vector spaces $S$ and $T$ satisfy $S\oplus T=\real^n$, then,
for every $x\in \real^n$, there exist unique $x_S\in S$ and $x_T\in T$
such that $x=x_S+x_T$; the vector $x_S$ is called the \emph{oblique
  projection of $x$ onto $S$ parallel to $T$} and the map
$\map{\prj}{\real^n}{S}$ defined by $\prj(x)=x_S$ is the \emph{oblique
  projection operator onto $S$ parallel to $T$}. If $T=S^{\perp}$,
then $\prj$ is the \emph{orthogonal projection onto $S$}.



An undirected weighted graph is a triple $G=(V,\mathcal{E},A)$, where
$V$ is the set of vertices with $|V|=n$, $\mathcal{E}\subseteq V\times
V$ is the set of edges with $|\mathcal{E}|=m$, and
$A=A^\top\in\real^{n\times n}$ is the nonnegative adjacency
matrix. The Laplacian $L\in\real^{n\times n}$ of $G$ is defined by
$L=\diag\left(\left\{\sum_{i=1}^{n} a_{ij}\right\}_{j\in
  V}\right)-A$. A path in $G$ is an ordered sequence of 
vertices such that there is an edge between every two consecutive
vertices. A path is \emph{simple} if no vertex appears more than once
in it, except possibly for the case when the initial and final
vertices are the same. A \emph{cycle} is a simple path that starts and
ends at the same vertex and has at least three vertices. If the graph $G$ is connected, then $\dim(\Img(L))=n-1$
\cite[Lemma 13.1.1]{CDG-GFR:01}. After choosing an enumeration and
orientation for the edges of $G$, we let $B\in \real^{n\times m}$
denote the oriented incidence matrix of $G$. For a connected $G$, we
have $\Img(B)=\vect{1}^{\perp}_n$. It is known that
$B\mathcal{A}B^{\top}=L$, where $\mathcal{A}\in \real^{m\times m}$ is
the diagonal weight matrix defined by $\mathcal{A}_{ef}=a_{ij}$ if
both edges $e$ and $f$ are equal to $(i,j)$, and $0$ otherwise. For every vector $\mathbf{v}\in \real^m$,
we define $L_{\mathbf{v}}\in \real^{n\times n}$ by $L_{\mathbf{v}}=B\mathcal{A}\diag(\mathbf{v})B^{\top}$.

\section{The Kuramoto model}\label{sec:kuramoto}

The Kuramoto model is a system of $n$ oscillators, where each
oscillator has a natural frequency $\omega_i\in \real$ and is
described by a phase angle $\theta_i\in \mathbb{S}^1$. The
interconnection of the oscillators is described by a weighted
undirected connected graph $G=(\until{n},\mathcal{E},A)$, with nodes
$\until{n}$, edges $\mathcal{E}\subseteq\until{n}\times\until{n}$, and
weights $a_{ij}=a_{ji}>0$.  The dynamics for the Kuramoto model is:
\begin{equation}\label{eq:2}
\dot{\theta}_i=\omega_i-\sum_{j=1}^{n}a_{ij}
\sin(\theta_i-\theta_j),\qquad\text{for } i\in\until{n}.
\end{equation}
In matrix language, using the incidence matrix $B$ associated to an
arbitrary orientation of the graph and the weight matrix $\mathcal{A}$, one
can write the differential equations~\eqref{eq:2} as:
\begin{equation}\label{eq:kuramoto_model}
\dot{\theta}=\omega-B\mathcal{A}\sin(B^{\top}\theta),
\end{equation}
where $\theta=(\theta_1,\theta_2,\ldots,\theta_n)^{\top}\in
\torus^n$ is the phase vector and
$\omega=(\omega_1,\omega_2,\ldots,\omega_n)^{\top}\in \real^n$ is the
natural frequency vector. For every $s\in [0,2\pi)$, the clockwise
  rotation of $\theta\in \torus^n$ by the angle $s$ is the
  function $\mathrm{rot}_s:\torus^n\to \torus^n$ defined by
\begin{equation*}
\mathrm{rot}_s(\theta)=(\theta_1+s,\theta_2+s,\ldots,\theta_n+s)^{\top},\qquad\text{for }
\theta\in \torus^n.
\end{equation*}
Given $\theta\in \torus^n$, define the equivalence class $[\theta]$ by
\begin{align*}
[\theta]=\left\{\mathrm{rot}_s(\theta)\mid s\in [0,2\pi)\right\}.
\end{align*}
The quotient space of $\torus^n$ under the above equivalence class
is denoted by $\torus^n/\mathrm{rot}$. If
$\map{\theta}{\mathbb{R}_{\ge0}}{\torus^n}$ is a solution for the
Kuramoto model~\eqref{eq:kuramoto_model} then, for every $s\in
[0,2\pi)$, the curve
  $\map{\mathrm{rot}_s(\theta)}{\real_{\ge0}}{\torus^n}$ is also a
  solution of~\eqref{eq:kuramoto_model}. Therefore, for the rest of
  this paper, we consider the state space of the Kuramoto
  model~\eqref{eq:kuramoto_model} to be
  $\torus^n/\mathrm{rot}$. 

\begin{definition}[\textbf{Frequency synchronization}]
\begin{enumerate}
\item A solution $\map{\theta}{\mathbb{R}_{\ge0}}{\torus^n}$ of
  the Kuramoto model~\eqref{eq:kuramoto_model} achieves
  \emph{frequency synchronization} if there exists a synchronous
  frequency function
  $\map{\subscr{\omega}{syn}}{\real_{\ge0}}{\real}$ such that
\begin{equation*}
\lim_{t\to\infty} \dot{\theta}(t)=\subscr{\omega}{syn}(t)\vect{1}_n.
\end{equation*}
\item For a subset $S$ of the torus $\torus^n$, the coupled oscillator
\eqref{eq:kuramoto_model} achieves \emph{frequency synchronization}
if, for every $\theta_0\in S$, the trajectory of
\eqref{eq:kuramoto_model} starting at $\theta_0$ achieves frequency
synchronization.
\end{enumerate}
\end{definition}
If a solution of the coupled oscillator~\eqref{eq:kuramoto_model} achieves
frequency synchronization, then by summing all the equations in
\eqref{eq:kuramoto_model} and taking the limit as $t\to\infty$, we obtain
      $\subscr{\omega}{syn}=\sum_{i=1}^{n} \omega_i/n$.
      Therefore, the synchronous frequency is constant and is equal to the
      average of the natural frequency of the oscillators. By choosing a
      rotating frame with the frequency $\frac{\subscr{\omega}{syn}}{n}$,
      one can assume that $\omega\in \vect{1}^{\perp}_n$.
\begin{definition}[\textbf{Synchronization manifold}]
Let $\theta^*$ be a solution of the algebraic equation
\begin{equation}\label{eq:synchronization_manifold}
\omega=B\mathcal{A}\sin(B^{\top}\theta^*).
\end{equation}
Then $[\theta^*]$ is a \emph{synchronization manifold} for the
Kuramoto model~\eqref{eq:kuramoto_model}.
\end{definition}
In other words, the synchronization manifolds of the Kuramoto
model~\eqref{eq:kuramoto_model} are the equilibrium manifolds of the
differential equations~\eqref{eq:kuramoto_model}.

\begin{theorem}[\textbf{Characterization of frequency synchronization}]\label{thm:9}
Consider the Kuramoto model~\eqref{eq:kuramoto_model}, with the graph
$G$, incidence matrix $B$, weight matrix $\mathcal{A}$, and natural
frequencies $\omega\in \vect{1}_n^{\perp}$. Then the following
statements are equivalent:
\begin{enumerate}
\item\label{p1:frequency} there exists an open set $U\subset\torus^n$
  such that every solution of the Kuramoto
  model~\eqref{eq:kuramoto_model} achieves frequency synchronization;
\item\label{p2:equilibrium} there exists a locally asymptotically
  stable synchronization manifold for~\eqref{eq:kuramoto_model}.
\end{enumerate} 
\end{theorem}


\section{Preliminary results}\label{sec:perliminary}


\subsection{The cutset projection associated to a weighted digraph}\label{subsec:oblique_projection}
We here introduce and study a useful oblique projection operator; to the
best of our knowledge, this operator and its graph theoretic interpretation
have not been studied previously. We start with some definitions for a
digraph $G$ with $n$ nodes and $m$ edges.  For a simple path $\gamma$ in $G$,
we define the \emph{signed weighted path vector} $v^\gamma\in\real^m$ of
the simple path $\gamma$ by
\begin{equation*}
  v_e^\gamma =
  \begin{cases}
    {+\displaystyle{1}/{a_{kl}}},\quad & \mbox{if the edge $e=(k,l)$ is traversed} \\[-.5ex]
    & \qquad\qquad\qquad\qquad\quad \mbox{positively by $\gamma$,} 
    \\
    {\displaystyle-{1}/{a_{kl}}}, \quad &  \mbox{if the edge $e=(k,l)$ is traversed  } \\[-.5ex]
    & \qquad\qquad\qquad\qquad\quad\mbox{negatively by $\gamma$,} \\
  0, & \mbox{otherwise}.
\end{cases}
\end{equation*} 
For a partition of the
vertices of $V$ in two non-empty disjoint sets $\phi$ and
$\phi^{\mathsf{c}}$, the \emph{cutset orientation vector}
corresponding to the partition $V=\phi\bigcup\phi^{\mathsf{c}}$ is the
vector $v^\phi\in\real^m$ given by
\begin{equation*}
  v_e^\phi=
  \begin{cases}
    +1, \quad & \mbox{if edge $e$ has its source node in $\phi$ } \\[-.5ex]
    &  \qquad\qquad\mbox{and its sink node in $\phi^{\mathsf{c}}$},
    \\
    -1, &\mbox{if edge $e$ has its source node in $\phi^{\mathsf{c}}$ } \\[-.5ex]
    & \qquad\qquad \mbox{and  its sink node in $\phi$},\\
    0, & \mbox{otherwise}.
  \end{cases}
\end{equation*}
The \emph{weighted cycle space} of $G$ is the subspace of $\real^m$ spanned
by the signed weighted path vectors of all simple undirected cycles in $G$.
(Note that the notion of cycle space is standard, while that
  of weighted cycle space is not.)  The \emph{cutset space} of $G$ is
subspace of $\real^m$ spanned by the cutset orientation vectors of all cuts
of the nodes of $G$. It is a variation of a known fact, e.g.,
  see~\cite[Theorem~8.5]{FB:18}, that
\begin{align*}
  & \text{weighted cycle space} \\
  &\enspace=
  \spn\setdef{v^\gamma\in\real^m}{\gamma \text{ is a simple cycle in
      $G$}} = \Ker(B\mathcal{A}),
  \\ &\text{cutset space} \\
  &\enspace=
  \spn\setdef{v^\phi\in\real^m}{\phi \text{ is a cut of $G$}} =
  \Img(B^{\top}).
\end{align*}

\begin{theorem}[\textbf{Decomposition of edge space and the cutset projection}]
  \label{thm:decomposition}
  Let $G$ be an undirected weighted connected graph with $n$ nodes and $m$
  edges, incidence matrix $B$, and weight matrix
  $\mathcal{A}$. Recall that the Laplacian of $G$ is given by
  $L=B\mathcal{A} B^\top$. Then
  \begin{enumerate}
  \item\label{p1:edge_decomposition} the edge space $\real^m$ can be
    decomposed as the direct sum:
    \begin{equation*}
      \real^m = \Img(B^{\top})\oplus\Ker(B\mathcal{A});
    \end{equation*}
    
  \item\label{p2:oblique_projection} the \emph{cutset projection matrix} $\prj$,
    defined to be the oblique projection onto $\Img(B^{\top})$ parallel to
    $\Ker(B\mathcal{A})$, is given by
    \begin{align}
      \label{def:prj}
      \prj=B^{\top}L^{\dagger}B\mathcal{A};
    \end{align}
    
  \item\label{p1:idempotent} the cutset projection matrix $\prj$ is idempotent,
    and $0$ and $1$ are eigenvalues with algebraic (and geometric) multiplicity
    $m-n+1$ and $n-1$, respectively.
\end{enumerate}
\end{theorem}

The proof of this theorem is given in
Appendix~\ref{app:decomposition}.  Recall that, given a full rank
matrix $C$, the orthogonal projection onto $\Img(C)$ is given by the
formula $C(C^{\top}C)^{-1}C^{\top}$.  As we show in
Appendix~\ref{app:decomposition}, the equality~\eqref{def:prj} is an
application of a generalized version of this formula.  Next, we
establish some properties of the cutset projection matrix, whose proof
is again postponed to Appendix~\ref{app:decomposition}.

\begin{theorem}[\textbf{Properties of the cutset projection}]\label{thm:projection_property}
    Consider an undirected weighted connected graph $G$ with incidence matrix
    $B$, weight matrix $\mathcal{A}$, and cutset projection matrix
    $\prj$. Then the following statements hold:
  \begin{enumerate} 
  \item\label{p2:orthogonal} if $G$ is unweighted (that is, $\mathcal{A}=I_m$),
    then $\prj$ is an orthogonal projection matrix and $\|\prj\|_2=1$;
    
  \item\label{p3:acyclic} if $G$ is acyclic, then $\prj=I_{m}$ and $\|\prj\|_{\infty}=1$;

  \item\label{p4:complete} if $G$ is an unweighted complete graph, then
    $\prj=\tfrac{1}{n}B^\top B$ and $\|\prj\|_{\infty}=\frac{2(n-1)}{n}$; and

  \item\label{p5:ring} if $G$ is an unweighted ring graph, then
    $\prj=I_n-\frac{1}{n}\vect{1}_n\vect{1}_n^{\top}$ and
    $\|\prj\|_{\infty}=\frac{2(n-1)}{n}$.
  \end{enumerate}
\end{theorem}

We conclude with some observations without proof.

\begin{remark}[Connection with effective resistances and minimal
  angle] With the same notation as in
  Theorem~\ref{thm:projection_property},
\begin{enumerate}
  \item the decomposition $\real^m=\Img(B^{\top})\oplus \Ker(B\mathcal{A})$
    and the cutset projection matrix $\prj$ depend on the 
      edge orientation chosen on $G$. However, it can be shown that, for
    every $p\in [1,\infty)\union \{\infty\}$, the induced norm
      $\|\prj\|_{p}$ is independent of the specific orientation;
      
  \item if $\R\in \real^{n\times n}$ is the matrix of effective
    resistances of the weighted graph $G$, then
    $\prj=-\frac{1}{2}B^{\top}\R B \mathcal{A}$~\cite{AG-SB-AS:08};
    and

  \item if $\theta$ is the minimal angle between the cutset space of
    $G$ and the weighted cycle space of $G$, then
    $\sin(\theta)=\|\prj\|^{-1}_{2}$~\cite[Theorem 3.1]{ICFI-CDM:95}.
\end{enumerate}
\end{remark}

\subsection{Embedded cohesive subset}\label{subsec:embedded_cohesive}

In this subsection, we introduce a new subset of the $n$-torus, called the
embedded cohesive subset. This subset plays an essential role in our
analysis of the Kuramoto model~\eqref{eq:kuramoto_model}. In
  what follows, recall that $\left|\theta_i-\theta_j\right|$ is the
  geodesic distance on $\torus$ between angles $\theta_i$ and $\theta_j$,
  as defined in equation~\eqref{def:geodesic-distance}.

\begin{definition}[\textbf{Arc subset, cohesive subset, and embedded cohesive subsets}]
Let $G$ be an undirected weighted connected graph with edge set
$\mathcal{E}$ and let $\gamma\in [0,\pi)$.
\begin{enumerate}
\item The \emph{arc subset} $\overline{\Gamma}(\gamma)\subset
  \torus^n$ is the set of $\theta\in\torus^n$ such that there
  exists an arc of length $\gamma$ in $\mathbb{S}^1$ containing all
  angles $\theta_1,\theta_2,\ldots,\theta_n$. The set $\Gamma(\gamma)$
  is the interior of $\overline{\Gamma}(\gamma)$;

\item The \emph{cohesive subset} $\Delta^G(\gamma)\subseteq
  \torus^n$ is 
  \begin{align*}
    \Delta^G(\gamma)=\setdef{\theta\in \torus^n}{
      \left|\theta_i-\theta_j\right|\le \gamma, \quad \text{for all } (i,j)\in \mathcal{E}};
  \end{align*}
  
\item The \emph{embedded cohesive subset} $S^{G}(\gamma)\subseteq
  \torus^n$ is 
  \begin{equation*}
    S^{G}(\gamma)=\setdef{\mathrm{rot}_s(\exp(\imagunit\mathbf{x}))}{\mathbf{x}\in\bperp \text{ and }s\in [0,2\pi)},
  \end{equation*}
  where we define $\bperp=\{\mathbf{x}\in
  \vect{1}_n^{\perp}\mid\|B^{\top}\mathbf{x}\|_{\infty}\le \gamma\}$.
\end{enumerate}
\end{definition}

It is easy to see that the arc subsets, the cohesive subsets, and the
embedded cohesive subset are invariant under the rotations
$\mathrm{rot}_s$, for every $s\in [0,2\pi)$. Therefore, in the rest of
  this paper, without any ambiguity, we use the notations
  $\Gamma(\gamma)$, $\Delta^G(\gamma)$, and $S^{G}(\gamma)$ for the
  set of equivalent classes of the arc subsets, the cohesive subset,
  and the embedded cohesive subset, respectively.

Note that it is clear how to check whether a point in $\torus^n$
belongs to the arc subset and/or the cohesive subset.  We next present
an algorithm, called the Embedding Algorithm, that allows one to
easily check whether a point in $\torus^n$ belongs to the embedded
cohesive set or not.

\smallskip

\begin{algorithm}
  \caption{\textbf{Embedding Algorithm}}
  \label{algo:embedding}

\begin{algorithmic}[1]
  \REQUIRE{$\theta\in\torus^n$}
  \STATE $x_1:=0$ and $S := \{1\}$
  \WHILE{$|S|<n$}
  \STATE pick $\{j,k\}\in\mathcal{E}$ s.t.\ $j\in{S}$ and  $k\in\until{n}\setminus{S}$  
  \IF{geodesic arc from $\theta_j$ to  $\theta_k$ is counterclockwise}
  \STATE $x_k:=x_j+|\theta_j-\theta_k|$
  \ELSE
  \STATE $x_k:=x_j-|\theta_j-\theta_k|$
  \ENDIF
  \STATE  $S := S \union \{k\}$
  \ENDWHILE
  \STATE  $\mathbf{x}^{\theta}_i := x_i - \operatorname{average}(x_1,\dots,x_n)$ for $i\in\until{n}$
  \RETURN $\mathbf{x}^{\theta}\in\vect{1}_n^\perp$ 
\end{algorithmic}
\end{algorithm}
\smallskip

  We now characterize the embedded cohesive set; the proof of the
  following theorem is given in Appendix~\ref{app:phase_cohesive}.

\begin{theorem}[\textbf{Characterization of the embedded cohesive subset}]\label{thm:embedded_cohesive_characterization}
  Let $G$ be an undirected weighted connected graph and $\gamma\in[0,\pi)$.
    For $\theta\in \torus^n$, let
      $\mathbf{x}^{\theta}\in\vect{1}_n^\perp$ be the corresponding output
      of the Embedding Algorithm. Then the following statements holds:
    \begin{enumerate}
    \item\label{p1:inclusion-chain} $\overline{\Gamma}(\gamma)\subseteq
      S^G(\gamma)\subseteq \Delta^{G}(\gamma)$;
    \item\label{p2:exp} $\exp(\imagunit\mathbf{x}^{\theta})\in [\theta]$;
    \item\label{p3:embeded_cohesive} $[\theta]\in S^{G}(\gamma)$ if and
      only if $\mathbf{x}^{\theta}\in \bperp$; and
    \item\label{p4:diffeomorphic_S} the set $S^{G}(\gamma)$ is
      diffeomorphic with $\bperp$ and the set $S^{G}(\gamma)$ is compact.
\end{enumerate}
\end{theorem}



\begin{figure}[!htb]\centering
\includegraphics[width=.85\linewidth]{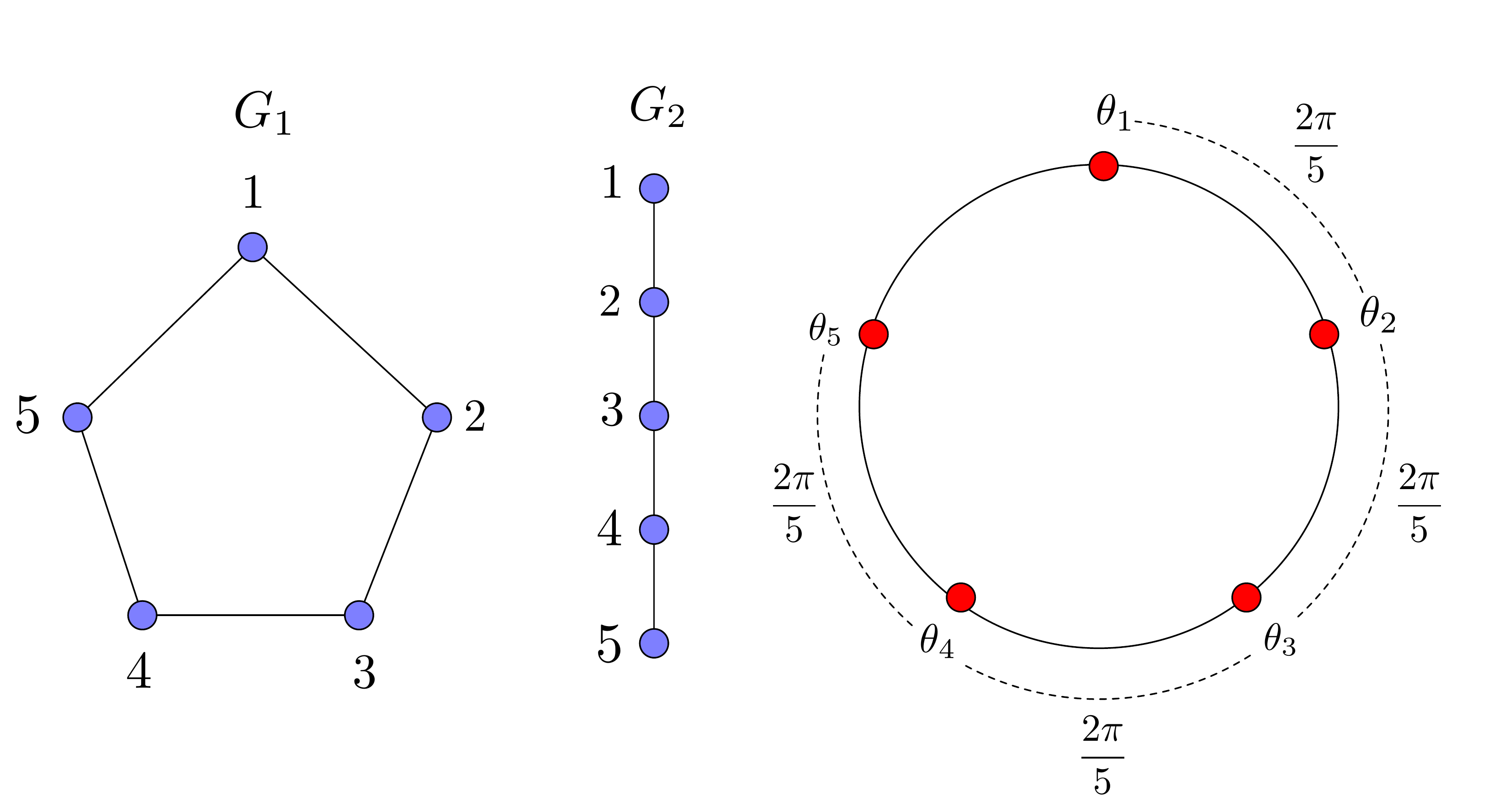}
\caption{For the graph $G_1$, the $5$-tuple defined by the red points
  on the circle belongs to
  the cohesive subset $\Delta^{G_1}(\tfrac{2\pi}{5})$ but it does not
  belong to 
  the embedded cohesive subset $S^{G_1}(\tfrac{2\pi}{5})$. For the
  graph $G_2$, this $5$-tuple belongs to the embedded cohesive subset
  $S^{G_2}(\tfrac{2\pi}{5})$ but it does not belong to the arc subset
  $\overline{\Gamma}(\tfrac{2\pi}{5})$.}
\label{fig:embed_cohesive}
\end{figure}

Based on
Theorem~\ref{thm:embedded_cohesive_characterization}\ref{p4:diffeomorphic_S},
in the rest of this paper, we identify the embedded cohesive subset
$S^G(\gamma)$ by the set $\bperp=\setdef{\mathbf{x}\in
  \vect{1}_n^{\perp}}{\|B^{\top}\mathbf{x}\|_{\infty}\le \gamma}$. We
conclude this subsection with an instructive comparison.

\begin{example}[\textbf{Comparing the three subsets}]
  Pick $\gamma\in [0,\pi)$. In this example, we show that each of the
    inclusions in
    Theorem~\ref{thm:embedded_cohesive_characterization}\ref{p1:inclusion-chain}
    is strict in general.
    \begin{enumerate}
    \item Consider a $5$-cycle graph $G_1$ with the vector
      $\theta=\begin{pmatrix} 0 &\frac{2\pi}{5} & \frac{4\pi}{5} &
      \frac{6\pi}{5} & \frac{8\pi}{5} \end{pmatrix}^{\top}$ as shown
      in Figure~\ref{fig:embed_cohesive}. One can verify that
      $|\theta_i-\theta_j|=\frac{2\pi}{5}$, for every $(i,j)\in
      \mathcal{E}$ and therefore $\theta\in
      \Delta^{G_1}(\gamma)$. However, using the embedding algorithm, it can be shown that
      $\theta\not\in S^{G_1}(\frac{2\pi}{5})$.
    \item Consider an acyclic graph $G_2$ with $5$ nodes and the vector
      $\theta=\begin{pmatrix} 0 &\frac{2\pi}{5} & \frac{4\pi}{5} &
      \frac{6\pi}{5} & \frac{8\pi}{5} \end{pmatrix}^{\top}$ as shown
      in Figure~\ref{fig:embed_cohesive}. Then it is clear that
      $\theta\not\in \Gamma(\frac{2\pi}{5})$. However, using the
      embedding algorithm, it can be
      shown that $\theta\in S^{G_2}(\frac{2\pi}{5})$.
\end{enumerate} 
\end{example}

\subsection{Kuramoto map and its properties}\label{subsec:kuramoto_map}
We now define the \emph{Kuramoto map}
$\map{\fK}{\vect{1}^{\perp}_n}{\Img(B^{\top})}$ by
\begin{equation}
  \label{def:Kuramoto-map}
  \fK (\mathbf{x})=\prj\sin(B^{\top}\mathbf{x}).
\end{equation}
This map arises naturally from the Kuramoto model as follows.  Recall
that, given nodal variables $\omega\in\vect{1}_n^\perp\subset\real^n$
and given the identification in
Theorem~\ref{thm:embedded_cohesive_characterization}\ref{p4:diffeomorphic_S},
the equilibrium equation~\eqref{eq:synchronization_manifold} can be
rewritten as 
\begin{equation} \label{eq:synchronization_manifold-x}
  \omega = B\mathcal{A}\sin(B^{\top}\mathbf{x})
\end{equation}
and can be interpreted as a nodal balance equation.  If one
left-multiplies this nodal balance equation by $B^{\top}L^{\dagger}$,
then one obtains an edge balance equation
\begin{equation}
  \label{eq:synchronization_manifold-edges}
  B^{\top}L^{\dagger}\omega=\prj \sin(B^{\top}\mathbf{x})=\fK (\mathbf{x}),
\end{equation}
where $B^{\top}L^{\dagger}\omega \in \Img(B^{\top})\subset\real^m$ can
be interpreted as a collection of flows through each edge.

The following theorem studies the properties of the map~$\fK$ and
shows the equivalence between the nodal and edge balance equations;
see Appendix~\ref{app:real_analytic} for the proof.

\begin{theorem}[\textbf{Basic properties of the Kuramoto map}]
  \label{thm:existence_uniqueness_S}
  Consider the Kuramoto model~\eqref{eq:kuramoto_model}, with the
  graph $G$, incidence matrix $B$, weight matrix $\mathcal{A}$, and
  natural frequencies $\omega\in \vect{1}_n^{\perp}$.  Define the
  Kuramoto map as in~\eqref{def:Kuramoto-map} and pick $\gamma\in
  [0,\frac{\pi}{2})$. Then the following statements hold:
\begin{enumerate}
\item\label{p4:equivalence} $\mathbf{x}^*$ is a synchronization
  manifold for the Kuramoto model~\eqref{eq:kuramoto_model} if and
  only if $\mathbf{x}^*$ solves the nodal balance
  equation~\eqref{eq:synchronization_manifold-x} if and only if
  $\mathbf{x}^*$ solves edge nodal balance
  equation~\eqref{eq:synchronization_manifold-edges};
\item\label{p3:f_one_to_one} the function $\fK$ is real analytic and
  one-to-one on $S^G(\gamma)$;
\item\label{p5:existence_uniqueness_S} if there exists a
  synchronization manifold $\mathbf{x}^*\in S^{G}(\gamma)$, then it is
  unique and locally exponentially stable.
\end{enumerate} 
\end{theorem}


\section{Sufficient conditions for synchronization}\label{sec:synchronization_conditions}
In this section, we present novel sufficient conditions for existence and
uniqueness of the synchronization manifold for the Kuramoto model in the domain
$S^{G}(\gamma)$, for $\gamma\in [0,\tfrac{\pi}{2})$. We start with a useful
definition.
\begin{definition}[\textbf{Minimum amplification
      factor for scaled projection}]\label{def:maf} Consider an
  undirected weighted connected graph with incidence matrix $B$,
  weight matrix $\mathcal{A}$, and cutset projection matrix
  $\prj$. For $\gamma\in [0,\tfrac{\pi}{2})$ and
    $p\in[1,\infty)\union\{\infty\}$, define
      \begin{enumerate}
      \item the domain
        $D_p(\gamma)=\setdef{\mathbf{y}\in\real^m}{\|\mathbf{y}\|_{p}\le\gamma}
        \intersection \Img(B^{\top})$;
      \item for $\mathbf{y}\in D_p(\gamma)$, the \emph{scaled cutset projection
        operator}
        $\map{\prj\diag(\sinc(\mathbf{y}))}{\Img(B^{\top})}{\Img(B^{\top})}$;
        and
      \item the \emph{minimum amplification factor} of the scaled cutset
        projection $\prj\diag(\sinc(\mathbf{y}))$ on $D_p(\gamma)$ by:
        \begin{align*}
          \alpha_p(\gamma)= \min_{\mathbf{y}\in
            D_p(\gamma)}\min_{\begin{array}{c}\scriptstyle \mathbf{z}\in \Img(B^{\top}) \\[-3pt]
              \scriptstyle
              \|\mathbf{z}\|_p=1\end{array}}\|\prj\diag(\sinc(\mathbf{y}))\mathbf{z}\|_{p}. 
        \end{align*}
      \end{enumerate}
\end{definition}

Note that $\alpha_p(\gamma)$ is well-defined because
$(\mathbf{y},\mathbf{z})\mapsto\prj\diag(\sinc(\mathbf{y}))\mathbf{z}$
is a continuous function over a compact set.  The proof of the
following lemma is given in Appendix~\ref{app:maf}.

\begin{lemma}[\textbf{The minimum amplification factor is non-zero}]\label{thm:maf-non-zero}
  With the same notation and under the same assumptions as Definition~\ref{def:maf}, 
  the minimum amplification factor of the scaled projection satisfies
  $\alpha_p(\gamma)>0$.
\end{lemma}

\subsection{Main results}\label{sec:main-results}

Now, we are ready to state the main results of this paper. We start with a
family of general conditions for synchronization of Kuramoto
model~\eqref{eq:kuramoto_model}.

\begin{theorem}[\textbf{General sufficient conditions for synchronization}]
  \label{thm:existence+uniqueness_cond-new}
  Consider the Kuramoto model~\eqref{eq:kuramoto_model} with undirected weighted
  connected graph $G$, the incidence matrix $B$, the weight matrix
  $\mathcal{A}$, the cutset projection $\prj$, and frequencies $\omega\in
  \vect{1}_n^{\perp}$. For $\gamma\in [0,\tfrac{\pi}{2})$, if there exists 
  $p\in[1,\infty)\union\{\infty\}$ such that
  \begin{align}
    \tag{T1}\label{test:p-norm-new}
    \|B^{\top}L^{\dagger}\omega\|_{p} \le \alpha_{p}(\gamma)\gamma, 
  \end{align}
    then there exists a unique locally exponentially stable
   synchronization manifold $\mathbf{x}^*$ for the Kuramoto
   model~\eqref{eq:kuramoto_model} in the domain $S^{G}(\gamma)$.
\end{theorem}

Note that one can generalize test~\eqref{test:p-norm-new} to the setting of
arbitrary sub-multiplicative norms, with the caveat that the solution may take
value outside $S^G(\gamma)$.

\begin{proof}[Proof of Theorem~\ref{thm:existence+uniqueness_cond-new}]
We first show that, for every $p\in [1,\infty)\union\{\infty\}$, we have
  $D_p(\gamma)\subseteq B^{\top}(S^G(\gamma))$. Suppose that $\mathbf{y}\in
  D_p(\gamma)$. Then, by definition of $D_p(\gamma)$, there exists
  $\xi\in\vect{1}_n^{\perp}$ such that $\mathbf{y}=B^{\top}\xi$ and
  $\|B^{\top}\xi\|_p\le \gamma$. Note that, for any vector $y$, the $p$-norm of
  $y$ is larger than or equal to the $\infty$-norm of $y$. This implies that
  $\|B^{\top}\xi\|_{\infty} \le \|B^{\top}\xi\|_p\le \gamma$. Therefore, by
  definition of $S^{G}(\gamma)$, we obtain $\xi\in S^{G}(\gamma)$ and, as a
  result, we have $\mathbf{y}=B^{\top}\xi\in B^{\top}(S^{G}(\gamma))$. Suppose
  that $\gamma\in [0,\tfrac{\pi}{2})$ and $\mathbf{x}\in S^{G}(\gamma)$. Then
  $\mathbf{x}$ is a synchronization manifold for the Kuramoto
  model~\eqref{eq:kuramoto_model} if and only if
\begin{align*}
\prj\diag(\sinc(B^{\top}\mathbf{x}))B^{\top}\mathbf{x}=B^{\top}L^{\dagger}\omega.
\end{align*}
For every $\mathbf{y}\in D_p(\gamma)$, define the map $\map{Q(\mathbf{y})}{\Img(B^{\top})}{\Img(B^{\top})}$ by
\begin{align*}
Q(\mathbf{y})(\mathbf{z})=\prj\diag(\sinc(\mathbf{y}))\mathbf{z}.
\end{align*}
The following lemma, whose proof is given in
Appendix~\ref{app:lem:Q-inverse}, studies some of the properties of
the map $Q(\mathbf{y})$.
\begin{lemma}\label{lem:Q-inverse}
For every $\mathbf{y}\in D_p(\gamma)$, then the map $Q(\mathbf{y})$ is
invertible and, for every $\mathbf{z}\in \Img(B^{\top})$,
\begin{align}
 \left(Q(\mathbf{y})\right)^{-1}\mathbf{z}=(B^{\top}L^{\dagger}_{\sinc(\mathbf{y})}B\mathcal{A})\mathbf{z}.
\end{align}
\end{lemma}

Now we get back to the proof of
Theorem~\ref{thm:existence+uniqueness_cond-new}. For every $p\in
[1,\infty)\cup\{\infty\}$ and every $\gamma\in [0,\tfrac{\pi}{2})$
  define the map $\map{h_{p}}{D_{p}(\gamma)}{\Img(B^{\top})}$ by
\begin{align*}
h_p(\mathbf{y})=\left(Q(\mathbf{y})\right)^{-1}(B^{\top}L^{\dagger}\omega).
\end{align*}
Note that by Lemma~\ref{lem:Q-inverse}, we have 
\begin{align*}
h_p(\mathbf{y})=(B^{\top}L^{\dagger}_{\sinc(\mathbf{y})} B\mathcal{A})(B^{\top}L^{\dagger}\omega)=B^{\top}L^{\dagger}_{\sinc(\mathbf{y})}\omega.
\end{align*}
For every $\mathbf{y}\in D_p(\gamma)$, we have
$\dim(\Img(L_{\sinc(\mathbf{y})}))=n-1$. Therefore,
by~\cite[Theorem 4.2]{VR:97}, the map $h_p$ is continuous on
$D_p(\gamma)$. We first show that, if the assumption~\eqref{test:p-norm-new} holds, then
$h_p(\gamma)\subseteq D_p(\gamma)$. Given $\omega\in \vect{1}_n^{\perp}$, note the
following inequality:
\begin{align}\label{eq:bound-hp}
\|h_p(\mathbf{y})\|_p&=\|\left(Q(\mathbf{y})\right)^{-1}(B^{\top}L^{\dagger}\omega)\|_p\\
               &\le \max_{\mathbf{y}\in
  D_p(\gamma)}\|\left(Q(\mathbf{y})\right)^{-1}\|_p\|B^{\top}L^{\dagger}\omega\|_p.\nonumber
\end{align}
Since $Q(\mathbf{y})$ is invertible, using
Lemma~\ref{thm:min-amplification} in Appendix~\ref{app:maf},
\begin{align}\label{eq:min-max}
\max_{\mathbf{y}\in
  D_p(\gamma)}\|\left(Q(\mathbf{y})\right)^{-1}\|_p&=\Big(\min_{\mathbf{y}\in D_p(\gamma)}
\min_{\substack{\mathbf{z}\in\Img(B^{\top}),\\
  \|\mathbf{z}\|_p=1}}\|Q(\mathbf{y})\mathbf{z}\|_p \Big)^{-1}\nonumber\\
                                                      & = \frac{1}{\alpha_p(\gamma)}.
\end{align}
Combining the inequalities~\eqref{eq:bound-hp},~\eqref{eq:min-max},
and~\eqref{test:p-norm-new}, we obtain
\begin{align*}
  \|h_p(\mathbf{y})\|_p\le
  \frac{1}{\alpha_p(\gamma)}\|B^{\top}L^{\dagger}\omega\|_p
  \le \frac{1}{\alpha_p(\gamma)}\alpha_p(\gamma)\gamma\le \gamma.
\end{align*}
In summary, $h_p$ is a continuous map from a compact convex set into
itself. Therefore, by the Brouwer Fixed-Point Theorem, $h_p$ has a
fixed-point in $D_{p}(\gamma)$. Since, for every $p\in
[1,\infty)\union\{\infty\}$, we have $D_p(\gamma)\subseteq
  B^{\top}(S^G(\gamma))$, there exists $\mathbf{x}\in S^G(\gamma)$ such
  that $h_p(B^{\top}\mathbf{x})=B^{\top}\mathbf{x}$. Therefore, we have
\begin{align*}
  B^{\top}L^{\dagger}\omega=\prj\sin(B^{\top}\mathbf{x}).
\end{align*}
The fact that $\mathbf{x}$ is the unique synchronization manifold of
the Kuramoto model~\eqref{eq:kuramoto_model} in $S^G(\gamma)$ follows
from Theorem~\ref{thm:existence_uniqueness_S} parts~\ref{p4:equivalence} and~\ref{p5:existence_uniqueness_S}.
\end{proof}

Theorem~\ref{thm:existence+uniqueness_cond-new} presents a novel
family of sufficient synchronization conditions for the Kuramoto
model~\eqref{eq:kuramoto_model}. However, these tests require the
computation of the minimum amplification factor $\alpha_p(\gamma)$,
that is, the solution to an optimization problem (see
Definition~\ref{def:maf}) that is generally nonconvex. At this time
we do not know of any reliable numerical method to compute $\alpha_p(\gamma)$ for large
dimensional systems. Therefore, in what follows, we focus on finding
explicit lower bounds on $\alpha_p(\gamma)$, thereby obtaining
computable synchronization tests.

\begin{theorem}[\textbf{Sufficient conditions for synchronization
      based on $2$-norm}]\label{thm:2-norm-computable} Consider the
  Kuramoto model~\eqref{eq:kuramoto_model} with undirected unweighted
  connected graph $G$, the incidence matrix $B$, the cutset projection matrix $\prj$, and frequencies $\omega\in
  \vect{1}_n^{\perp}$. Then the following statements hold:
   \begin{enumerate}
   \item\label{p1:2-norm-maf} for every $\gamma\in [0,\frac{\pi}{2})$, 
   \begin{align*}
   \alpha_2(\gamma)\ge \sinc(\gamma);
   \end{align*}
   \item\label{p2:2-norm-new} for every $\gamma\in
       [0,\frac{\pi}{2})$, if the following condition holds: 
   \begin{align}
   \tag{T2}\label{test:2-norm-new}
   \|B^{\top}L^{\dagger}\omega\|_2&\le\sin(\gamma),
   \end{align}
   then there exists a unique locally exponentially stable
   synchronization manifold $\mathbf{x}^*$ for the Kuramoto
   model~\eqref{eq:kuramoto_model} in the domain $S^{G}(\gamma)$.
   \end{enumerate}
   \end{theorem}

\begin{proof}
First note that $\diag(\sinc(\mathbf{x}))\succeq \sinc(\gamma)I_m$,
for every $\mathbf{x}\in D_2(\gamma)$. This implies that
\begin{align}\label{eq:ineq:L}
L_{\sinc(\mathbf{x})} = B \diag(\sinc(\mathbf{x})) B^{\top} \succeq
  \sinc(\gamma)L. 
\end{align}
Multiplying both sides of the inequality~\eqref{eq:ineq:L} by
$(L^{\dagger})^{\frac{1}{2}}$, we get
\begin{align*}
(L^{\dagger})^{\frac{1}{2}}L_{\sinc(\mathbf{x})}
  (L^{\dagger})^{\frac{1}{2}} \succeq \sinc(\gamma)(I_n - \tfrac{1}{n}\vect{1}_n\vect{1}_n^{\perp}). 
\end{align*}
Note that $\lambda_i(I_n - \tfrac{1}{n}\vect{1}_n\vect{1}_n^{\perp}) = 1$,
for every $i\in \{2,\ldots,n\}$. Thus, \cite[Corollary 7.7.4
  (c)]{RAH-CRJ:12} implies
\begin{align*}
\lambda_i((L^{\dagger})^{\frac{1}{2}}L_{\sinc(\mathbf{x})}
  (L^{\dagger})^{\frac{1}{2}}) \ge \sinc(\gamma), \qquad i\in \{2,\ldots,n\}.
\end{align*}
In turn this inequality implies
\begin{align*}
  \lambda_i(L^{\frac{1}{2}}L^{\dagger}_{\sinc(\mathbf{x})}L^{\frac{1}{2}}) \le \tfrac{1}{\sinc(\gamma)}, \qquad i\in \{2,\ldots,n\}.
\end{align*}
Now, Weyl's Theorem~\cite[Theorem 4.3.1]{RAH-CRJ:12} implies
\begin{align*}
\tfrac{1}{\sinc(\gamma)}(I_n - \tfrac{1}{n}\vect{1}_n\vect{1}_n^{\perp}) \succeq L^{\frac{1}{2}}L^{\dagger}_{\sinc(\mathbf{x})}L^{\frac{1}{2}} ,
\end{align*}
so that, for every $\mathbf{x}\in D_2(\gamma)$,
\begin{align}\label{eq:ineq:pinvL}
  L^{\dagger} \succeq \sinc(\gamma)L^{\dagger}_{\sinc(\mathbf{x})}.
\end{align}
Regarding part~\ref{p1:2-norm-maf}, note that $\prj = B^{\top}L^{\dagger}B$
is an idempotent symmetric matrix. Thus, by setting $\mathbf{y} =
B^{\top}\mathbf{w}$, we have
\begin{multline*}
\|\prj\diag(\sinc(\mathbf{x}))\mathbf{y}\|^2_2\\= \mathbf{w}^{\top}B \diag(\sinc(\mathbf{x})) B^{\top}
  L^{\dagger} B \diag(\sinc(\mathbf{x})) B^{\top} \mathbf{w}.
\end{multline*}
Therefore, using~\eqref{eq:ineq:L} and~\eqref{eq:ineq:pinvL}, we get 
\begin{align*}
\|\prj&\diag(\sinc(\mathbf{x}))\mathbf{y}\|^2_2\\ &= \mathbf{w}^{\top}B \diag(\sinc(\mathbf{x})) B^{\top}
  L^{\dagger} B \diag(\sinc(\mathbf{x})) B^{\top} \mathbf{w} \\ &\ge
  \sinc(\gamma) \mathbf{w}^{\top}B \diag(\sinc(\mathbf{x})) B^{\top}
  \mathbf{w} \\ & \ge \sinc(\gamma)^2 \mathbf{w}^{\top}BB^{\top}\mathbf{w}
  = \sinc(\gamma)^2 \mathbf{y}^{\top}\mathbf{y} = \sinc(\gamma)^2.
\end{align*}
This completes the proof of part~\ref{p1:2-norm-maf}.  Regarding
part~\ref{p2:2-norm-new}, if the assumption~\eqref{test:2-norm-new} holds,
then
\begin{align*}
  \|B^{\top}L^{\dagger}\omega\|_{2}\le
  \sin(\gamma)=\sinc(\gamma)\gamma\le \alpha_2(\gamma)\gamma.
\end{align*}
The result follows by using the test~\eqref{test:p-norm-new} for $p=2$. 
\end{proof}



It is now convenient to introduce the smooth function
$\map{g}{[1,\infty)}{\real}$ defined by
  \begin{multline}\label{eq:g-function}
    g(x)= \frac{y(x)+\sin(y(x))}{2} \\ - x \frac{y(x)-\sin(y(x))}{2}
    \,\Big|_{y(x) = \arccos(\frac{x-1}{x+1})},
  \end{multline}
One can  verify that $g(1)=1$, $g$ is monotonically decreasing,
  and $\lim_{x\to\infty} g(x)=0$; the graph of $g$ is shown in
  Figure~\ref{fig:function-g}.

\begin{figure}[!htb]\centering
  \includegraphics[width=.75\linewidth]{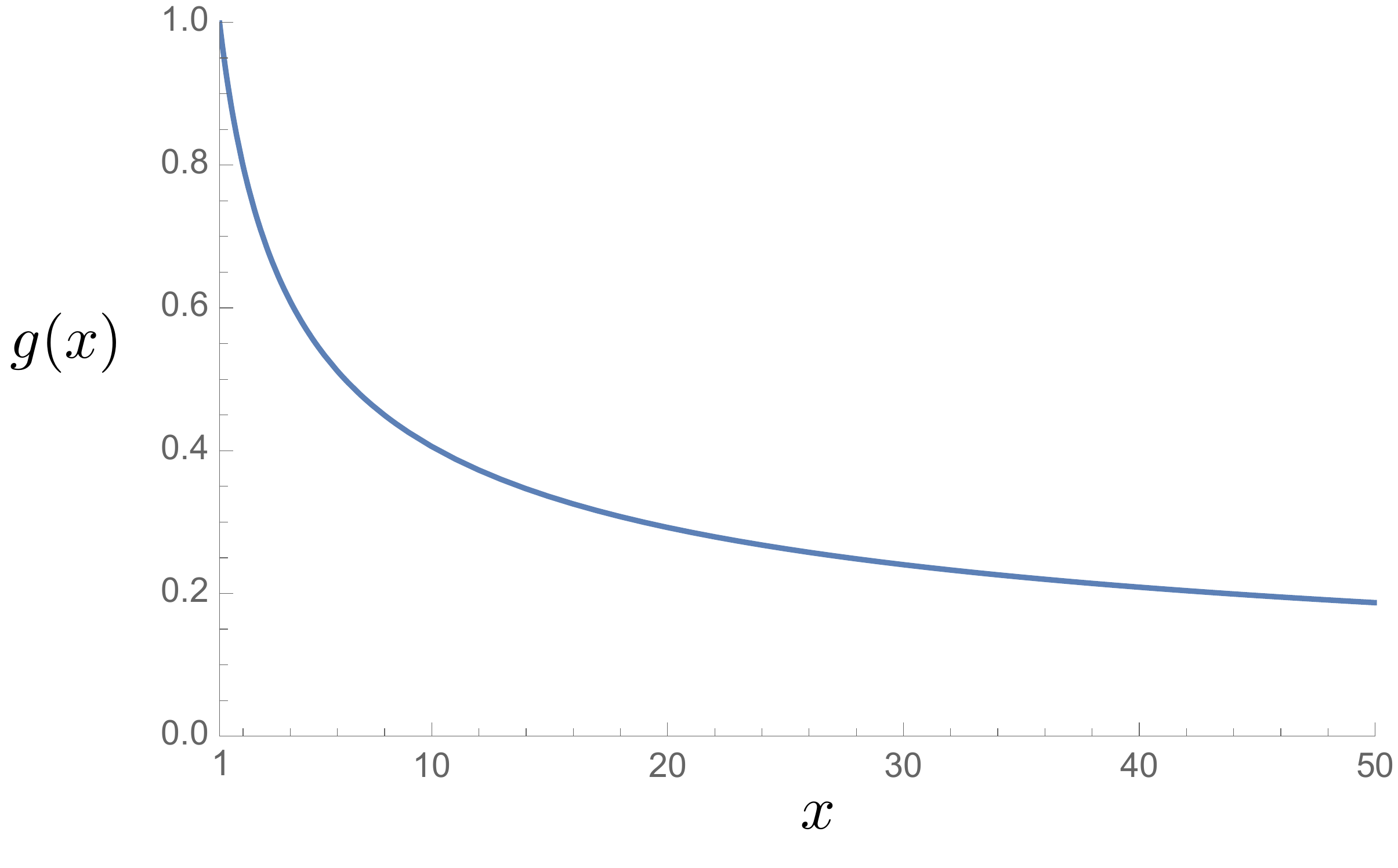}
  \caption{The graph of the monotonically-decreasing function $g$}
  \label{fig:function-g}
\end{figure}

\begin{theorem}[\textbf{Sufficient conditions for synchronization
      based on general lower bound}]\label{thm:p-norm-computable}
  Consider the Kuramoto model~\eqref{eq:kuramoto_model} with
  undirected weighted connected graph
  $G$, the incidence matrix $B$, the weight matrix $\mathcal{A}$, the
  cutset projection matrix $\prj$, and frequencies $\omega\in
  \vect{1}_n^{\perp}$. For
  $p\in[1,\infty)\union\{\infty\}$, define the angle
    \begin{align}
      \label{def:gammapstar}
   \gamma_p^*=\arccos\left(\frac{\|\prj\|_p-1}{\|\prj\|_p+1}\right) \in
    [0,\tfrac{\pi}{2}). 
   \end{align}
   Then the following statements hold:
   \begin{enumerate}
   \item\label{p1:p-norm_maf} for every $\gamma\in [0,\frac{\pi}{2})$,
     we have
   \begin{align*}
     \alpha_p(\gamma) \ge
     \left(\frac{1+\sinc(\gamma)}{2}\right)-\|\prj\|_{p}\left(\frac{1-\sinc(\gamma)}{2}\right);
   \end{align*}
   \item\label{p2:p-norm-computable} if the following condition holds:
   \begin{align}\label{test:p-norm-computable}
   \tag{T3}\|B^{\top}L^{\dagger}\omega\|_{p}\le g(\|\prj\|_p), 
   \end{align}
   then there exists a unique locally exponentially stable
   synchronization manifold $\mathbf{x}^*$ for the Kuramoto
   model~\eqref{eq:kuramoto_model} in the domain $S^{G}(\gamma_p^*)$.
   \end{enumerate}
\end{theorem}

\begin{proof}
Regarding part~\ref{p1:p-norm_maf}, let $\mathbf{w}\in \real^m$.  The triangle inequality implies that,
for every $\mathbf{y}\in D_{p}(\gamma)$ and every $\mathbf{z}\in
\Img(B^{\top})$ with $\|\mathbf{z}\|_p=1$,
\begin{multline}\label{eq:triangular_inequality}
\left\|\prj \diag(\sinc(\mathbf{y}))\mathbf{z}\right\|_{p}\\ \ge
  \left\|\prj
  \diag(\mathbf{w})\mathbf{z}\right\|_{p}-\left\|\prj
  \diag(\sinc(\mathbf{y})-\mathbf{w})\mathbf{z}\right\|_{p}. 
\end{multline}
Using triangle inequality, the last term
  in the inequality~\eqref{eq:triangular_inequality} can be upper bounded as
\begin{align*}
\left\|\prj \diag(\sinc(\mathbf{y})-\mathbf{w})\mathbf{z}\right\|_{p} \le \|\prj\|_{p}\|\diag(\sinc(\mathbf{y})-\mathbf{w})\|_{p},
\end{align*}
Moreover, the matrix $\diag(\sinc(\mathbf{y})-\mathbf{w})$ is diagonal
and by~\cite[Theorem 5.6.36]{RAH-CRJ:12}, we have
\begin{align*}
\|\diag(\sinc(\mathbf{y})-\mathbf{w})\|_p = \|\sinc(\mathbf{y})-\mathbf{w}\|_{\infty}.
\end{align*} 
Therefore, the inequality~\eqref{eq:triangular_inequality} can be rewritten as
\begin{multline*}
  \left\|\prj \diag(\sinc(\mathbf{y}))\mathbf{z}\right\|_{p}\\\ge
  \left\|\prj
  \diag(\mathbf{w})\mathbf{z}\right\|_{p}-\|\prj\|_{p}\|\sinc(\mathbf{y})-\mathbf{w}\|_{\infty}. 
\end{multline*}
By setting $\mathbf{w}=\left(\frac{1+\sinc(\gamma)}{2}\right)\vect{1}_m$, we have
\begin{align*}
\|\prj\diag(\mathbf{w})\mathbf{z}\|_p=\left(\frac{1+\sinc(\gamma)}{2}\right)\|\prj\mathbf{z}\|_p=\frac{1+\sinc(\gamma)}{2}.
\end{align*}
In turn $\|\sinc(\mathbf{y})-\mathbf{w}\|_{\infty}\le
\frac{1-\sinc(\gamma)}{2}$, and we get
\begin{multline*} 
  \left\|\prj \diag(\sinc(\mathbf{y}))\mathbf{z}\right\|_{p}\\
  \ge
  \left(\frac{1+\sinc(\gamma)}{2}\right)-\|\prj\|_{p}\left(\frac{1-\sinc(\gamma)}{2}\right). 
\end{multline*}
Part~\ref{p1:p-norm_maf} of the theorem simply follows by taking the
minimum over $\mathbf{y}\in D_p(\gamma)$ and $\mathbf{z}\in
\Img(B^{\top})$ such that $\|\mathbf{z}\|_{p}=1$.

Regarding part~\ref{p2:p-norm-computable}, note that, by
part~\ref{p1:p-norm_maf}, we have
\begin{align*}
\alpha_p(\gamma)\ge \left(\frac{\gamma+\sin(\gamma)}{2}\right)-\|\prj\|_p\left(\frac{\gamma-\sin(\gamma)}{2}\right).
\end{align*}
Define the function $\map{\ell}{[0,\tfrac{\pi}{2})}{\real}$ by
\begin{align*}
\ell(\gamma)=\left(\frac{\gamma+\sin(\gamma)}{2}\right)-\|\prj\|_p\left(\frac{\gamma-\sin(\gamma)}{2}\right).
\end{align*}
Then one can compute:
\begin{align}
\frac{d \ell}{d\gamma}(\gamma)&=\left(\frac{1+\cos(\gamma)}{2}\right)-\|\prj\|_p\left(\frac{1-\cos(\gamma)}{2}\right),\label{eq:first_der_f}\\
\frac{d^2 \ell}{d\gamma^2}(\gamma)&=\left(\frac{-\sin(\gamma)}{2}\right)-\|\prj\|_p\left(\frac{\sin(\gamma)}{2}\right).\label{eq:second_der_f}
\end{align} 
Using the equation~\eqref{eq:first_der_f}, one can check that the unique critical point of $\ell$ in the interval
$[0,\tfrac{\pi}{2})$ is the solution $\gamma_p^*$ to
\begin{align*}
\left(\frac{1+\cos(\gamma^*_p)}{2}\right)-\|\prj\|_p\left(\frac{1-\cos(\gamma^*_p)}{2}\right)=0.
\end{align*}
This implies that $\gamma_p^*\in [0,\frac{\pi}{2})$ is given as in
equation~\eqref{def:gammapstar}.  Moreover, we have
\begin{align*}
\frac{d^2}{d\gamma^2}\ell(\gamma^*_p)=\left(\frac{-\sin(\gamma^*_p)}{2}\right)-\|\prj\|_p\left(\frac{\sin(\gamma^*_p)}{2}\right)<0,
\end{align*}
so that $\gamma=\gamma_p^*$ is a local maximum for~$\ell$. Now by using test~\eqref{test:p-norm-new}, if the following condition holds:
\begin{align*}
\|B^{\top}L^{\dagger}\omega\|_{p}\le \alpha_p(\gamma^*_p)\gamma^*_p,
\end{align*}
then there exists a unique locally exponentially stable
synchronization manifold $\mathbf{x}^*$ for the Kuramoto
model~\eqref{eq:kuramoto_model} in the domain $S^{G}(\gamma^*_p)$. Using
the lower bound for $\alpha_p(\gamma)$ given in
part~\ref{p1:p-norm_maf}, it is easy to see that if the following
condition holds:
\begin{align*}
\|B^{\top}L^{\dagger}\omega\|_{p}& \le
  \left(\frac{\gamma_p^*+\sin(\gamma_p^*)}{2}\right)-\|\prj\|_p\left(\frac{\gamma_p^*-\sin(\gamma_p^*)}{2}\right)\\
  & =g(\|\prj\|_p), 
\end{align*}
then there exists a unique locally exponentially stable
synchronization manifold $\mathbf{x}^*$ for the Kuramoto
model~\eqref{eq:kuramoto_model} in the domain $S^{G}(\gamma^*_p)$. This completes the proof of the theorem.
\end{proof}

\begin{remark}[Comparison of test~\eqref{test:2-norm-new}
  and test~\eqref{test:p-norm-computable}]
\begin{enumerate}
\item The tests~\eqref{test:2-norm-new}
  and~\eqref{test:p-norm-computable} have different domain of
  applicability. For every $\gamma\in [0,\frac{\pi}{2})$, the
  test~\eqref{test:2-norm-new} presents a sufficient condition for
  synchronization of the Kuramoto model~\eqref{eq:kuramoto_model} in
  $S^{G}(\gamma)$. Instead, the test~\eqref{test:p-norm-computable} is
  applicable only to the specific domain $S^{G}(\gamma^*_p)$ for
  $\gamma^*_p=\arccos\left(\frac{\|\prj\|_p-1}{\|\prj\|_p+1}\right)\in
  [0,\tfrac{\pi}{2})$.
  
 \item For unweighted graphs with $\gamma=\frac{\pi}{2}$, one can
   recover test~\eqref{test:2-norm-new} from
   test~\eqref{test:p-norm-computable}. Specifically, for unweighted
   graphs, Theorem~\ref{thm:projection_property}\ref{p2:orthogonal}
   implies that $\gamma^*_2=\arccos(0)=\frac{\pi}{2}$.  Thus, for
   unweighted graphs, test~\eqref{test:p-norm-computable} with $p=2$
   is $\|B^{\top}L^{\dagger}\omega\|_2\le
   g(1)=\sin(\tfrac{\pi}{2})=1$.
  \end{enumerate}
  \end{remark}


\subsection{Comparison with previously-known synchronization results}\label{sec:comparison}

We now compare the new synchronization tests~\eqref{test:2-norm-new}
and~\eqref{test:p-norm-computable} with those existing in the
literature.

  \subsubsection{General topology ($2$-norm synchronization conditions)}

    To the best of our knowledge, sufficient
    conditions for synchronization of networks of oscillators with general topology
    was first studied in the paper~\cite{AJ-NM-MB:04}. Using
    the analysis methods introduced in~\cite{AJ-NM-MB:04}, the tightest sufficient condition for synchronization of networks of
    oscillators with general topology~\cite[Theorem 4.7]{FD-FB:12i} can be obtained by the following test:
    \begin{align}\label{test:best-known}
      \tag{T0}\left\|B^{\top}\omega\right\|_2< \lambda_2(L),
    \end{align}
    where $\lambda_2(L)$ is the Fiedler eigenvalue of the Laplacian $L$
    (see the survey~\cite{FD-FB:13b} for more discussion). One can show
    that, for unweighted graphs test~\eqref{test:2-norm-new} gives a sharper sufficient condition
    than test~\eqref{test:best-known}.  This fact is a consequence of the
    following lemma, whose proof is given in Appendix~\ref{app:inequality}.
    \begin{lemma}\label{lem:inequality} Let $G$ be a
        connected, undirected, weighted graph with the incidence matrix $B$
        and weight matrix $\mathcal{A}$. Assume that $L$ is the Laplacian
        of $G$ with eigenvalues
        $0=\lambda_1(L)<\lambda_2(L)\le\cdots\le\lambda_n(L)$. Then the
        following statements hold:
   \begin{enumerate}
   \item\label{p1:inequality} each $\omega\in \vect{1}_n^{\perp}$satisfies the inequality
     \begin{align}\label{eq:inequality}
       \left\| B^{\top}L^{\dagger}\omega \right\|_2\le
       \frac{1}{\lambda_2(L)}\left\|B^{\top}\omega \right\|_2,
     \end{align}
     with the equality sign if and only if $\omega$ belongs to the
     eigenspace associated to $\lambda_2(L)$; and
   \item\label{p2:T0_T2} if $\omega\in \vect{1}_n^{\perp}$ satisfies
     test~\eqref{test:best-known}, then it satisfies
     test~\eqref{test:2-norm-new}.
   \end{enumerate}
   \end{lemma}

  
\subsubsection{General topology ($\infty$-norm synchronization conditions)}
  The approximate test $\|B^{\top}L^{\dagger}\omega\|_{\infty}\le 1$
  was proposed in~\cite{FD-MC-FB:11v-pnas} as an approximately-correct
  sufficient condition for synchronization; statistical evidence on
  random graphs and IEEE test cases shows that the condition has much
  predictive power. However, \cite{FD-MC-FB:11v-pnas} also identifies
  a family of counterexamples, where the condition is shown to be
  incorrect. Our test~\eqref{test:p-norm-computable} with $p=\infty$
  is a rigorous, more conservative, and generically-applicable version
  of that approximately-correct test.

\subsubsection{Acyclic topology}

Consider the Kuramoto model~\eqref{eq:kuramoto_model} with acyclic connected graph
$G$ and $\omega\in \vect{1}_n^{\perp}$. Then the existence and uniqueness of synchronization
manifolds of~\eqref{eq:kuramoto_model} in $S^{G}(\gamma)$ can be completely characterized by the
following test~\cite[Corollary 7.5]{FD-FB:13b}: 
\begin{align}\label{eq:acyclic}
\|B^{\top}L^{\dagger}\omega\|_{\infty}\le \sin(\gamma). 
\end{align}
We show that this characterization can be obtained from
the general test~\eqref{test:p-norm-new} for $p=\infty$.

\begin{corollary}[\textbf{Synchronization for acyclic graphs}]\label{cor:acyclic}
Consider the Kuramoto model~\eqref{eq:kuramoto_model} with the acyclic undirected weighted connected
graph $G$, the incidence matrix $B$, the weight matrix $A$, and $\omega\in \vect{1}_n^{\perp}$. Pick
$\gamma\in [0,\tfrac{\pi}{2})$. Then the following conditions are
  equivalent:
\begin{enumerate}
\item\label{p1:nec_suf_condition_acyclic} $\left\|B^{\top}L^{\dagger}\omega\right\|_{\infty}\le \sin(\gamma)$,
\item\label{p2:state_existence_uniqueness_acyclic} there exists a
  unique locally exponentially stable synchronization manifold for the
  Kuramoto model~\eqref{eq:kuramoto_model} in $S^{G}(\gamma)$.
\end{enumerate}
Additionally, if either of the above equivalent conditions holds, then
the unique synchronization manifold in $S^{G}(\gamma)$ is
$L^{\dagger}B\mathcal{A}\arcsin(B^{\top}L^{\dagger}\omega)$.
\end{corollary}

\begin{proof}
Since $G$ is acyclic, Theorem~\ref{thm:projection_property}\ref{p3:acyclic} implies
that $\prj=B^{\top}L^{\dagger}B\mathcal{A}=I_m=I_{n-1}$.  

\ref{p1:nec_suf_condition_acyclic}
$\Longrightarrow$~\ref{p2:state_existence_uniqueness_acyclic}: Note
that, for every $\mathbf{y}\in D_{\infty}(\gamma)$ and every
$\mathbf{z}\in \Img(B^{\top})$ such that $\|\mathbf{z}\|_{\infty}=1$, we have 
\begin{align*}
\left\|\diag(\sinc(\mathbf{y}))\mathbf{z}\right\|_{\infty}\ge \sinc(\gamma)\|\mathbf{z}\|_{\infty}=\sinc(\gamma).
\end{align*}
In turn, this implies that
\begin{align*}
\alpha_{\infty}(\gamma)&=\min_{\mathbf{y}\in D_{\infty}(\gamma)}\min_{\substack{\mathbf{z}\in \Img(B^{\top}), \\ \|\mathbf{z}\|_{\infty}=1}} \left\|\prj\diag(\sinc(\mathbf{y}))\mathbf{z}\right\|_{\infty}\\ &=\min_{\mathbf{y}\in D_{\infty}(\gamma)} \min_{\substack{\mathbf{z}\in \Img(B^{\top}), \\ \|\mathbf{z}\|_{\infty}=1}}\left\|\diag(\sinc(\mathbf{y}))\mathbf{z}\right\|_{\infty}\\ & = \sinc(\gamma).
\end{align*}
Therefore, by using the test~\eqref{test:p-norm-new} for $p=\infty$, there exists a unique
locally stable synchronization manifold for the Kuramoto
model~\eqref{eq:kuramoto_model}.

\ref{p2:state_existence_uniqueness_acyclic}
$\Longrightarrow$~\ref{p1:nec_suf_condition_acyclic}: If
there exists a unique locally stable synchronization manifold
$\mathbf{x}^*$ for the Kuramoto model~\eqref{eq:kuramoto_model} in
$S^{G}(\gamma)$, by Theorem~\ref{thm:existence_uniqueness_S}\ref{p4:equivalence}, we have
\begin{align*}
  B^{\top}L^{\dagger}\omega=\prj \sin(B^{\top}\mathbf{x}^*)=\sin(B^{\top}\mathbf{x}^*).
\end{align*}
Since $\mathbf{x}^*\in S^{G}(\gamma)$ and $\gamma\in [0,\tfrac{\pi}{2})$,
we have $\left\|\sin(B^{\top}\mathbf{x}^*)\right\|_{\infty}\le
\sin(\gamma)$ and, in turn, 
\begin{align*}
  \left\|B^{\top}L^{\dagger}\omega\right\|_{\infty}\le \sin(\gamma).
\end{align*}
This completes the proof of equivalence of~\ref{p1:nec_suf_condition_acyclic}
and~\ref{p2:state_existence_uniqueness_acyclic}. If $\mathbf{x}^*$ is a synchronization manifold for the Kuramoto model~\eqref{eq:kuramoto_model} in
$S^{G}(\gamma)$, then we have $B^{\top}L^{\dagger}\omega=\prj \sin(B^{\top}\mathbf{x}^*)=\sin(B^{\top}\mathbf{x}^*)$.
This implies that
\begin{align*}
B^{\top}\mathbf{x}^*=\arcsin(B^{\top}L^{\dagger}\omega).
\end{align*}
Therefore, by pre-multiplying both side of the above equality into $B\mathcal{A}$,
we obtain
\begin{align*}
  B\mathcal{A}B^{\top}\mathbf{x}^*=L\mathbf{x}^*=B\mathcal{A}\arcsin(B^{\top}L^{\dagger}\omega).
\end{align*}
Thus, since $\mathbf{x}^*\in \vect{1}_n^{\perp}$, we get
$\mathbf{x}^*=L^{\dagger}B\mathcal{A}\arcsin(B^{\top}L^{\dagger}\omega)$. This
completes the proof of the theorem. 
\end{proof}

   \begin{remark}[Alternative way to recover acyclic case]
   One can prove Corollary~\ref{cor:acyclic}, using the lower bounds
   given in
   Theorem~\ref{thm:p-norm-computable}\ref{p1:p-norm_maf}. Note that,
   for acyclic graphs,
   Theorems~\ref{thm:projection_property}\ref{p3:acyclic}
   and~\ref{thm:p-norm-computable}\ref{p1:p-norm_maf} imply that
   $\alpha_\infty(\gamma)\ge \sinc(\gamma)$.  Combining this bound
   with the test~\eqref{test:p-norm-new} for $p=\infty$, we obtain
   $\|B^{\top}L^{\dagger}\omega\|_{\infty}\le \sin(\gamma)$.
   \end{remark}

\subsubsection{Unweighted ring graphs and unweighted complete graphs}

To the best of our knowledge, the
sharpest sufficient condition for existence of a synchronization
manifold in the
domain $S^{G}(\gamma)$ for the Kuramoto
model~\eqref{eq:kuramoto_model} with unweighted complete graph $G$ is given by the following test~\cite[Theorem
  6.6]{FD-FB:13b}:
\begin{align}\label{eq:best-complete}
  \|B^{\top}L^{\dagger}\omega\|_{\infty}\le \sin(\gamma), 
\end{align}
and for the Kuramoto model~\eqref{eq:kuramoto_model} with unweighted
ring graph $G$ is given by the following test~\cite[Theorem 3, Condition
  3]{FD-MC-FB:11v-pnas}:
\begin{align}\label{eq:best-ring}
  \|B^{\top}L^{\dagger}\omega\|_{\infty}\le \frac{1}{2}\sin(\gamma).  
\end{align} 

One can use the lower bound given in
Theorem~\ref{thm:p-norm-computable}\ref{p1:p-norm_maf} to obtain
another sufficient conditions for synchronization of unweighted complete
and ring graphs.

\begin{corollary}[\textbf{Sufficient synchronization conditions for unweighted
  complete and ring graphs}]
Consider the Kuramoto model~\eqref{eq:kuramoto_model} with either
unweighted complete or unweighted ring graph $G$, the incidence matrix $B$, cutset
projection $\prj$, and $\omega\in \vect{1}_n^{\perp}$. For every $n\in \N$, define the scalar function $\map{h_n}{[0,\tfrac{\pi}{2})}{\real_{\ge0}}$ by
\begin{align}\label{eq:def:h_n}
  h_n(\gamma)= \sin(\gamma)-\frac{n-2}{2n}\big(\gamma-\sin(\gamma)\big).
\end{align}
If the following condition holds: 
\begin{align*}
\|B^{\top}L^{\dagger}\omega\|_{\infty}\le h_n(\gamma), 
\end{align*}
then there exists a unique locally stable synchronization manifold
$\mathbf{x}^*$ for the Kuramoto model~\eqref{eq:kuramoto_model} in
$S^{G}(\gamma)$ 
\end{corollary}
\begin{proof}
By Theorem~\ref{thm:projection_property}\ref{p4:complete} and~\ref{p5:ring}, the
$\infty$-norm of the cutset projection matrix $\prj$ for graph $G$
with $n$ nodes which is either unweighted complete or unweighted ring, is given by $\|\prj\|_{\infty}=\frac{2(n-1)}{n}$. Therefore,
Theorem~\ref{thm:p-norm-computable}\ref{p1:p-norm_maf} implies the
following lower bound:
\begin{align}
  \label{eq:francesco-added-label-on-plane}
  \alpha_\infty(\gamma) \ge \sinc(\gamma)-\frac{n-2}{2n}\big(1-\sinc(\gamma)\big).
\end{align}
Combining the bound on $\alpha_\infty(\gamma)$ with the
test~\eqref{test:p-norm-new} for $p=\infty$, we get the following test
for the synchronization of unweighted complete graphs:
\begin{align}\label{eq:complete-graph-bound}
  \|B^{\top}L^{\dagger}\omega\|_{\infty}
  \le \sin(\gamma)-\frac{n-2}{2n}\big(\gamma-\sin(\gamma)\big)= h_n(\gamma).
\end{align}
The proof of the corollary is complete by using Theorem~\ref{thm:existence+uniqueness_cond-new}. 
\end{proof}
Note that, the function $\gamma\mapsto \left(\gamma-\sin(\gamma)\right)$ is positive and increasing on the
interval $\gamma\in (0,\tfrac{\pi}{2}]$. Therefore, for every $k<n$
and every $\gamma\in [0,\tfrac{\pi}{2})$, we have $h_n(\gamma)\le
h_k(\gamma)\le \sin(\gamma)$. Thus, for unweighted complete graphs,
the test~\eqref{eq:complete-graph-bound} is more conservative than the
existing test~\eqref{eq:best-complete}. It is worth mentioning that the gap between the new
test~\eqref{eq:complete-graph-bound} and the existing
test~\eqref{eq:best-complete} decreases with the decrease of the angle
$\gamma$. For instance, for $\gamma=\frac{\pi}{2}$, the right hand
side of the test~\eqref{eq:complete-graph-bound} asymptotically
converges to $\frac{3}{2}-\frac{\pi}{2}\approx0.715<1$. Therefore, at
$\gamma=\tfrac{\pi}{2}$, the sufficient
test~\eqref{eq:complete-graph-bound} is approximately $28.50~\%$ more
conservative than the test~\eqref{eq:best-complete}. Instead, for
$\gamma=\frac{\pi}{4}$, the right hand side of the
test~\eqref{eq:complete-graph-bound} asymptotically converges to
$\frac{3\sqrt{2}}{4}-\frac{\pi}{8}\approx0.668<\sin(\tfrac{\pi}{4})=\frac{\sqrt{2}}{2}\approx
0.707$. This means that, at $\gamma=\tfrac{\pi}{4}$, the
test~\eqref{eq:complete-graph-bound} is approximately $5.53~\%$ more
conservative than the test~\eqref{eq:best-complete}. The comparison
between the graph of the functions $h_5(x)$, $h_{10}(x)$, and $h_{20}(x)$
and $\sin(x)$ over the interval $[0,\tfrac{\pi}{2})$ is shown in
Figure~\ref{fig:complete_ring_comparison}.
   
For unweighted ring graphs, it is easy to see that the sufficient
condition~\eqref{eq:complete-graph-bound} is always sharper
than the existing sufficient condition~\eqref{eq:best-ring}. The
comparison between the graph of the functions $h_5(x)$, $h_{10}(x)$, and
$h_{20}(x)$ and $\tfrac{1}{2}\sin(x)$ for $x\in[0,\tfrac{\pi}{2})$ is
shown in Figure~\ref{fig:complete_ring_comparison}.


\begin{figure}[!htb]

\centering
        \includegraphics[width=.75\linewidth]{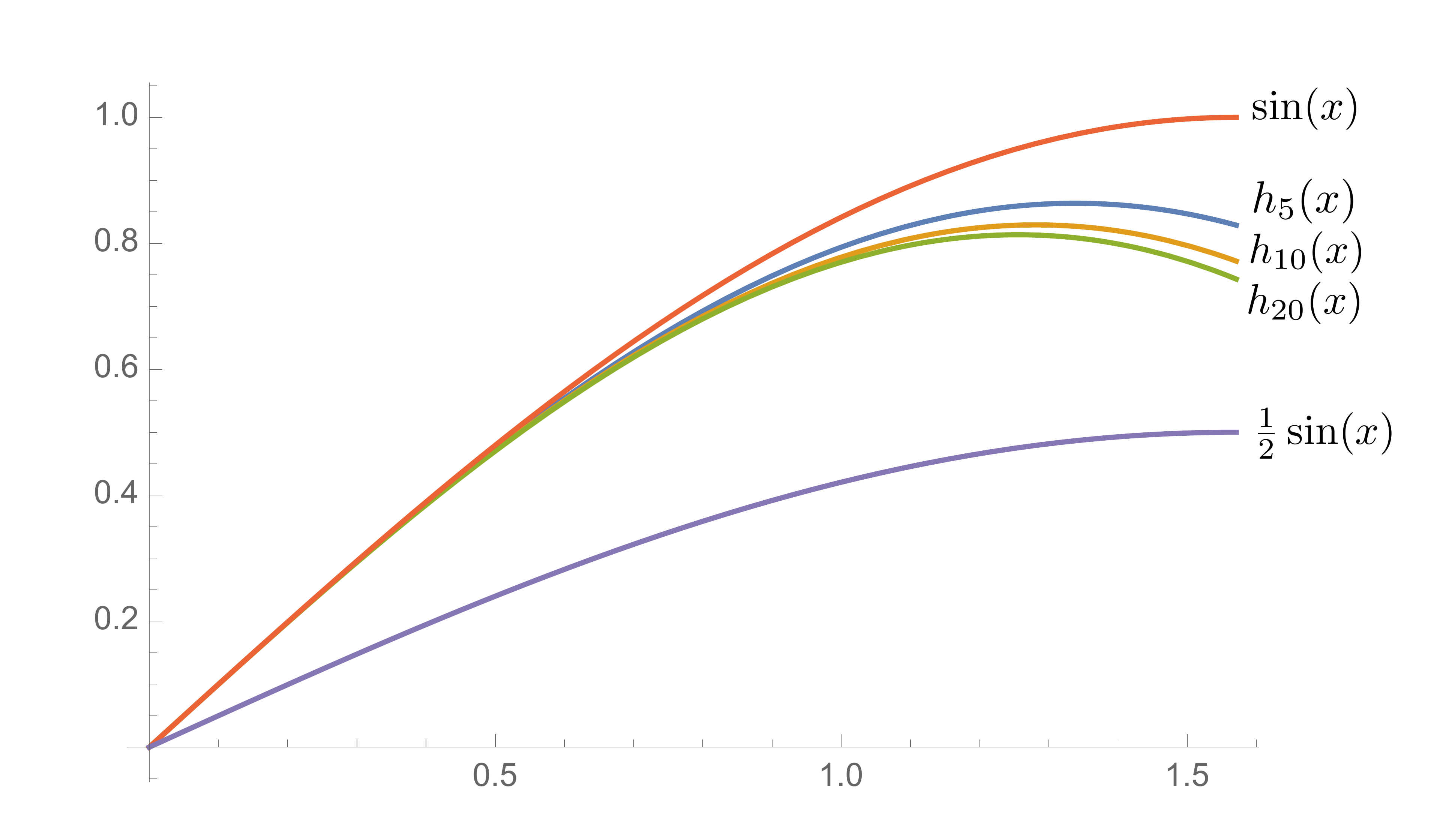}
        \caption{Comparison of the new sufficient tests with the existing
         sufficient tests for unweighted complete graphs and unweighted ring graphs.} \label{fig:complete_ring_comparison}
\end{figure}

\subsubsection{IEEE test cases}
Here we consider various IEEE test cases described by a connected
graph $G$ and a nodal admittance matrix $Y\in \mathbb{C}^{n\times
  n}$. The set of nodes of $G$ is partitioned into a set of load buses
$\mathcal{V}_1$ and a set of generator buses $\mathcal{V}_2$. The
voltage at the node $j\in \mathcal{V}_1\cup\mathcal{V}_2$ is denoted
by $V_j$, where $V_j= |V_j|e^{\imagunit\theta_j}$ and the power demand
(resp. power injection) at node $j\in \mathcal{V}_1$ (resp. $j\in \mathcal{V}_2$)
is denoted by $P_j$. By ignoring the resistances in the network, the
synchronization manifold $[\theta]$ of the network satisfies the
following Kuramoto model~\cite{FD-MC-FB:11v-pnas}:
\begin{align}\label{eq:power-network}
P_j - \sum\nolimits_{l\in \mathcal{V}_1\cup \mathcal{V}_2} a_{jl}\sin(\theta_j-\theta_l)=0,\quad\forall
  j\in \mathcal{V}_1\cup \mathcal{V}_2. 
\end{align}
where $a_{jl}=a_{lj}=|V_j||V_l|\Im(Y_{jl})$, for connected nodes $j$
and $l$. For the nine IEEE test cases given in
Table~\ref{tab:IEEE-test-cases}, we numerically check the existence of a
synchronization manifold for the Kuramoto
model~\eqref{eq:power-network} in the domain $S^{G}({\pi}/{2})$. We
consider effective power injections to be a scalar multiplication of
nominal power injections, i.e., given nominal injections
$P^{\mathrm{nom}}$ we set $P_j=K P^{\mathrm{nom}}_j$, for some $K\in
\real_{>0}$ and for every $j\in \mathcal{V}_1\cup\mathcal{V}_2$. The
voltage magnitudes at the generator buses are pre-determined and the
voltage magnitudes at load buses are computed by solving the reactive
power balance equations using the optimal power flow solver provided
by MATPOWER~\cite{RDZ-CEMS-RJT:11}. The critical coupling of the
Kuramoto model~\eqref{eq:power-network} is denoted by $\subscr{K}{c}$
and is computed using MATLAB \emph{fsolve}. For a given test T, the
smallest value of scaling factor for which the test T fails is denoted
by $\subscr{K}{T}$. We define the critical ratio of the test T by
${\subscr{K}{T}}/{\subscr{K}{c}}$. Intuitively speaking, the critical
ratio shows the accuracy of the test~T. Table~\ref{tab:IEEE-test-cases}
contains the following information:
\begin{description}
\item[The first two columns] contain the critical ratio of the prior
  test~\eqref{test:best-known} from the literature and the new
  sufficient test~\eqref{test:p-norm-computable} proposed in this paper.
\item[The third column] gives the critical ratio for the following
  approximate test proposed in~\cite{FD-MC-FB:11v-pnas}:
  \begin{align}\label{test:Flo_approximate}
    \tag{AT0} \|B^{\top}L^{\dagger}\omega\|_{\infty}\le 1.
  \end{align}

\item[The fourth column] gives the critical ratio for following
  approximate version of the test~\eqref{test:p-norm-new}:
  \begin{align}\label{test:T1_approximate_fmincon}
    \tag{AT1} \|B^{\top}L^{\dagger}\omega\|_{\infty}\le
    ({\pi}/{2}) \alpha_{\infty}^*({\pi}/{2}).
  \end{align}
  where $\alpha_{\infty}^*({\pi}/{2})$ is the approximate value
  for $\alpha_{\infty}({\pi}/{2})$ computed using MATLAB
  \emph{fmincon}. Test~\eqref{test:T1_approximate_fmincon} is
  approximate since, in general, \emph{fmincon} may not converge to a
  solution and, even when it converges, the solution is only guaranteed
  to be an upper bound for $\alpha_{\infty}({\pi}/{2})$.
\end{description}
\begin{center}
\begin{table}[htb] 
\resizebox{0.99\linewidth}{!}{
\begin{tabular}{|l|c|c|c|c|}
\hline
\multirow{3}{*}{Test Case}  
			&\multicolumn{4}{c|}{Critical ratio $\subscr{K}{T}/\subscr{K}{c}$}\\
                         \cline{2-5}
&$\lambda_2$ test &$\infty$-norm test & Approx.
                                                         test
                                                       &General test
  \\
  &\eqref{test:best-known}\cite{FD-FB:12i}&\eqref{test:p-norm-computable}%
  &\eqref{test:Flo_approximate}~\cite{FD-MC-FB:11v-pnas}&\eqref{test:T1_approximate_fmincon}\\
\hline
\rowcolor{Gray}
IEEE 9  & 16.54~$\%$ & 73.74~$\%$& 92.13~$\%$& 85.06~$\%$\textsuperscript{\dag} \\

IEEE 14  & 8.33~$\%$& 59.42~$\%$& 83.09~$\%$& 81.32~$\%$\textsuperscript{\dag}\\

\rowcolor{Gray}
IEEE RTS 24 &3.86~$\%$& 53.44~$\%$&89.48~$\%$&89.48~$\%$\textsuperscript{\dag}\\

IEEE 30 &2.70~$\%$&55.70~$\%$&85.54~$\%$&85.54~$\%$\textsuperscript{\dag}\\

\rowcolor{Gray}
IEEE 39 &2.97~$\%$&67.57~$\%$&100~$\%$&100~$\%$\textsuperscript{\dag}\\

IEEE 57 &0.36~$\%$&40.69~$\%$&84.67~$\%$&\textemdash\textsuperscript{*}\\

\rowcolor{Gray}
IEEE 118 &0.29~$\%$&43.70~$\%$&85.95~$\%$&\textemdash\textsuperscript{*}\\

IEEE 300 & 0.20~$\%$&40.33~$\%$&99.80~$\%$&\textemdash\textsuperscript{*}\\
\rowcolor{Gray}
Polish 2383 &0.11~$\%$&29.08~$\%$&82.85~$\%$&\textemdash\textsuperscript{*}\\
\hline
\end{tabular}}
\begin{tablenotes}\footnotesize
                          \item \textsuperscript{\dag} \emph{fmincon}
                          has been run for $100$ randomized initial
                          phase angles. \item \textsuperscript{*} \emph{fmincon} does not
                          converge.
\end{tablenotes}
\caption{Comparison of sufficient and approximate synchronizations
  tests on IEEE test cases in the domain
  $S^{G}({\pi}/{2})$.}\label{tab:IEEE-test-cases}
\end{table}
\end{center}
Note how (i) our ordering \eqref{test:best-known} $<$ \eqref{test:p-norm-computable} $<$
\eqref{test:Flo_approximate} is representative of the tests' accuracy,
(ii) our proposed test \eqref{test:p-norm-computable} is two order of
magnitude more accurate that best-known prior
test~\eqref{test:best-known} in the larger test cases, (iii) the two
approximate tests \eqref{test:Flo_approximate} and
\eqref{test:T1_approximate_fmincon} are comparable (but our proposed
test \eqref{test:T1_approximate_fmincon} is much more computationally
complex).

\section{Conclusion}
In this paper, we introduced and studied the cutset projection, as a
geometric operator associated to a weighted undirected graph. This
operator naturally appears in the study of networks of Kuramoto
oscillators~\eqref{eq:kuramoto_model}; using this operator, we
obtained new families of sufficient conditions for network
synchronization. For a network of Kuramoto oscillators with incidence
matrix $B$, Laplacian $L$ and frequencies $\omega$, these sufficient
conditions are in the form of upper bounds on the $p$-norm of the edge
flow quantity $B^{\top}L^{\dagger}\omega$. In other words, our results
highlight the important role of this edge flow quantity in the
synchronization of Kuramoto oscillators.  We show that our results
significantly improve the existing sufficient conditions in the
literature in general and, specifically, for a number of IEEE power
network test cases.

Our approach and results suggest many future research
directions. First, it is important to study the cutset projection in
more detail and for more special cases. We envision that the cutset
projection and its properties will be a valuable tool in the study of
network flow systems, above and beyond the case of Kuramoto
oscillators. Secondly, it is of interest to analyze and improve the
accuracy of our sufficient conditions. This can be potentially done by
designing efficient algorithm for numerical computation (or
estimation) of the minimum amplification factor for large
graphs. Thirdly, it is interesting to compare our new $p$-norm tests,
for different $p\in [0,\infty)\cup\{\infty\}$, and potentially extend
  them using more general norms.  Finally, in power network
  applications, it is potentially of significance to generalize our
  novel approach to study the coupled power flow equations.

\section*{Acknowledgments}
The authors thank Elizabeth Huang for her helpful comments, and Brian
Johnson and Sairaj Dhople for wide-ranging discussions on power
networks. The second author would also like to thank Florian D\"orfler
for countless conversations on Kuramoto oscillators.

\appendices

\section{Proof of Theorems~\ref{thm:decomposition}
  and~\ref{thm:projection_property}}\label{app:decomposition}

We report here a useful well-known lemma, which is a simplified
version of~\cite[Theorem~13]{ABI-TNEG:03}.
\begin{lemma}[Oblique projections]\label{thm:13}
  For $m,n\in \N$, assume the matrices $X,Y\in \real^{n\times m}$
  satisfy $\Img(X)\oplus \Ker(Y^{\top})=\real^n$. Then the
  oblique projection matrix onto $\Img(X)$ parallel to $\Ker(Y^{\top})$ is
  $X(Y^{\top}X)^{\dagger}Y^{\top}$.
\end{lemma}

We are now in a position to prove Theorems~\ref{thm:decomposition}
and~\ref{thm:projection_property}.

\begin{proof}[Proof of Theorem~\ref{thm:decomposition}]
Regarding the statement~\ref{p1:edge_decomposition}, we first show that
$\Img(B^{\top})\bigcap \Ker(B\mathcal{A})=\{\vect{0}_m\}$. Suppose that $v\in
\Img(B^{\top})\bigcap\Ker(B\mathcal{A})$.  Then there exists $\xi\in \real^n$
such that $v=B^{\top}\xi$ and
\begin{equation*}
B\mathcal{A}v=B\mathcal{A}B^{\top}\xi=L\xi=\vect{0}_n.
\end{equation*}
Since $G$ is connected, $0$ is a simple eigenvalue of the Laplacian
$L$ associated to the eigenvector $\vect{1}_n$. This implies that $\xi\in
\mathrm{span}\{\vect{1}_{n}\}$ and $v=B^{\top}\xi=\vect{0}_m$. Therefore, $\Img(B^{\top})\bigcap\Ker(B\mathcal{A})=\{\vect{0}_m\}$. 
Moreover, note that:
\begin{align*}
\dim(\Img(B^{\top}))&=n-1, \\
  \dim(\Ker(B\mathcal{A}))&=\dim(\Ker(B))=m-\dim(\Img(B))\\
  & =m-n+1.
\end{align*}
Therefore, $\real^m=\Img(B^{\top})\oplus \Ker(B\mathcal{A})$ as in
statement~\ref{p1:edge_decomposition}.

Statement~\ref{p2:oblique_projection} of the theorem follows directly from
Lemma~\ref{thm:13} with $X=B^\top$ and $Y=\mathcal{A}B^\top$.

Finally, statement~\ref{p1:idempotent} is a known consequence of
statements~\ref{p2:oblique_projection} and~\ref{p1:edge_decomposition}. It is
instructive, anyway, to provide an independent proof. Note the following
equalities:
\begin{align*}
\left(B^{\top}L^{\dagger}B\mathcal{A}\right)\left(B^{\top}L^{\dagger}B\mathcal{A}\right)&
=B^{\top}L^{\dagger}(B\mathcal{A}B^{\top})L^{\dagger}B\mathcal{A}\\ &
=B^{\top}(L^{\dagger}L) L^{\dagger}B\mathcal{A}.
\end{align*}
Using the fact that
$L^{\dagger}L=LL^{\dagger}=I_n-\frac{1}{n}\vect{1}_n\vect{1}_{n}^{\perp}$,
we obtain $B^{\top}L^{\dagger}L=B^{\top}$. This implies that 
\begin{align*}
\left(B^{\top}L^{\dagger}B\mathcal{A}\right)\left(B^{\top}L^{\dagger}B\mathcal{A}\right)=B^{\top}L^{\dagger}B\mathcal{A}. 
\end{align*}
Thus, the cutset projection $\prj$ is an idempotent matrix and its only
eigenvalues are $0$ and $1$~\cite[1.1.P5]{RAH-CRJ:12}.
\end{proof}

\begin{proof}[Proof of Theorem~\ref{thm:projection_property}]
Regarding part~\ref{p2:orthogonal}, when $\mathcal{A}=I_m$, we have
$\Ker(B\mathcal{A})=\Ker(B)$. Moreover, for every
$x\in \Img(B^{\top})$, there exists $\alpha\in
\vect{1}_n^{\perp}$ such that $B^{\top}\alpha=x$. Therefore, for every $y\in \Ker(B)$,
\begin{align*}
x^{\top}y=(B^{\top}\alpha)^{\top}y=\alpha^{\top}By=0.
\end{align*}
Moreover, we have $\dim(\Img(B^{\top}))+\dim(\Ker(B))=m$. This
implies that $\Img(B^{\top})=\left(\Ker(B)\right)^{\perp}$. Therefore,
the projection $\prj=B^{\top}L^{\dagger}B$ is an orthogonal
projection.

Regarding part~\ref{p3:acyclic}, since $G$ is acyclic, we have
$|\mathcal{E}|=n-1$ and $\Img(B^{\top})=\real^{n-1}$. Now consider a vector $x\in
\real^{n-1}$. Since  $\Img(B^{\top})=\real^{n-1}$, there exists a unique $\alpha\in
\vect{1}^{\perp}_{n}$ such that $x=B^{\top}\alpha$. Therefore, we have
\begin{equation*}
B^{\top}L^{\dagger}B\mathcal{A}(x)=B^{\top}L^{\dagger}B\mathcal{A}(B^{\top}\alpha)=B^{\top}L^{\dagger}L\alpha=B^{\top}\alpha=x.
\end{equation*}
This implies that $\prj=I_{n-1}=I_m$.

Regarding part~\ref{p4:complete}, note that for a unweighted complete
graph, we have
$L=n\left(I_n-\frac{1}{n}\vect{1}_n\vect{1}_n^{\top}\right)$ and
$L^{\dagger}=\frac{1}{n}\left(I_n-\frac{1}{n}\vect{1}_n\vect{1}_n^{\top}\right)$~\cite[Lemma
III.13]{FD-FB:11d}. We compute the cutset projection
$\prj$ for an unweighted complete graph as follows:
\begin{align*}
  \prj=B^{\top}L^{\dagger}B\mathcal{A}=\frac{1}{n}B^{\top}\left(I_n-\frac{1}{n}\vect{1}_n\vect{1}_n^{\top}\right)B\left(I_m\right)=\frac{1}{n}B^{\top}B.
\end{align*}
Recalling the meaning of the columns of $B$, we compute, for any two edges $e,f\in \mathcal{E}$,
\begin{equation*}
  \prj_{ef}= 
  \begin{cases}
    \tfrac{2}{n}, \quad & \text{if }e=f,\\
    \tfrac{1}{n}, & \text{if } e\ne f \text{ and $e$ and $f$ originate
      or terminate } \\[-.5ex]
    & \qquad\qquad\qquad\text{at the same node},
    \\
    \tfrac{-1}{n}, & \text{if } e\ne f \text{ and $e$ originates where $f$ terminates } \\
    &\qquad\qquad\qquad\text{or vice versa,} \\
    0, & \text{if } e \mbox{ and } f \mbox{ have no node in common}.
 \end{cases}
 \end{equation*} 
Using this expression, it is easy to see that 
 $\|\prj\|_{\infty}=\frac{2(n-1)}{n}$.

Regarding part~\ref{p5:ring}, first note that, for an unweighted ring
graph, one can choose the orientation of $G$ such that
\begin{align*}
\Img(B^{\top})=\vect{1}_n^{\perp},\qquad
\Ker(B)=\mathrm{span}\{\vect{1}_n\}.
\end{align*}
Moreover, $G$ is unweighted and, by part~\ref{p2:orthogonal}, the
cutset projection $\prj$ is an orthogonal projection onto
$\Img(B^{\top})$. This implies that
$\prj=I_n-\frac{1}{n}\vect{1}_n\vect{1}_n^{\top}$ and simple
bookkeeping shows that $\|\prj\|_{\infty}=(1-1/n)+(n-1)/n=2(n-1)/n$.
\end{proof}

\section{Proof of
  Theorem~\ref{thm:embedded_cohesive_characterization}}\label{app:phase_cohesive}

We start with a preliminary result.
\begin{lemma}\label{lem:existence_path}
  Given $\theta\in \Delta^{G}(\gamma)$, let
  $\mathbf{x}^{\theta}=\begin{pmatrix}x_1 & \cdots &
  x_n \end{pmatrix}^{\perp}\in \vect{1}_n^{\perp}$ be the output of the
  Embedding Algorithm with input $\theta$. Then, for every two nodes
  $a,b\in\until{n}$, there exists a simple path $(i_1,i_2,\ldots ,i_r)$ in $G$
  such that $i_1=a$, $i_r=b$, and
  \begin{align*}
    |x_{i_k}-x_{i_{k-1}}|\le\gamma,\qquad\text{for all } k\in\{2,\ldots,r\}.
  \end{align*}
\end{lemma}
\begin{proof}
It suffices to show that, for every
$a\in \{2,\ldots,n\}$ there exists a
simple path $(1,i_2\ldots i_r)$ from the node $1$ to node $a$ such that, for every $k$, we have
\begin{align*}
|x_{i_k}-x_{i_{k-1}}|\le\gamma.
\end{align*}
Start with node $a$. Suppose that $a$ is being added to $S$ in the
Embedding algorithm in the $r$th iteration. For every $k\in
\{1,\ldots,r\}$, we denote the set $S$ in the $k$th iteration of the
Embedding algorithm by $S_k$. Therefore $a\in S_{r}$ but
$a\not\in S_{r-1}$. Thus, by the algorithm, there exists
$i_{r-1}\in S_{r-1}$ such that $|x_{i_{r-1}}-x_{i}|\le \gamma$. Now,
we can repeat this procedure for $i_{r-1}$ to get $i_{r-2}$ such that $|x_{i_{r-1}}-x_{i}|\le \gamma$. We can continue doing this
procedure until we get to the node $1$. Thus we get the simple path $(1,i_1,\ldots,i_r=a)$.
\end{proof}

We are now ready to provide the main proof of interest.

\begin{proof}[Proof of
  Theorem~\ref{thm:embedded_cohesive_characterization}]
Regarding part~\ref{p1:inclusion-chain}, we first show the inclusion
$S^{G}(\gamma)\subseteq \Delta^G(\gamma)$. Let $\theta\in
S^{G}(\gamma)$. Then, by definition of $S^{G}(\gamma)$, there exists
$\mathbf{x}\in \vect{1}_n^{\perp}$ and $s\in [0,2\pi)$ such that 
\begin{align*}
\|B^{\top}\mathbf{x}\|_{\infty}\le \gamma, \qquad
 [\theta]=[\exp(\imagunit \mathbf{x})]. 
\end{align*} 
This means that, for every $(i,j)\in\mathcal{E}$, we have
\begin{align*}
|\theta_i-\theta_j|=|\exp(\imagunit
  x_i)-\exp(\imagunit x_j)|\le |x_i-x_j|\le \gamma.
\end{align*}
Therefore, $[\theta]\in \Delta^G(\gamma)$. 

Now consider a point $\theta\in \overline{\Gamma}(\gamma)$. By definition, there exists an
arc of length $\gamma$ on $\mathbb{S}^{1}$ which contains
$\theta_1,\ldots,\theta_n$. Since $0\le \gamma<\pi$, this implies that
there exists $\mathbf{y}\in \real^n$ such that $\theta=\exp(\imagunit \mathbf{y})$ and, for every
$(i,j)\in\mathcal{E}$, we have
\begin{align*}
|\theta_i-\theta_j|=|y_i-y_j|.
\end{align*}
We define the vector
$\mathbf{x}\in \vect{1}_n^{\perp}$ by $\mathbf{x}=\mathbf{y}-\mathrm{average}(\mathbf{y})\vect{1}_n$.
Then it is clear that $\mathbf{x}\in \vect{1}_n^{\perp}$ and
$[\exp(\imagunit \mathbf{x})]=[\exp(\imagunit \mathbf{y})]=[\theta]$. On the other hand, for every
$(i,j)\in \mathcal{E}$, the distance between $x_i$ and $x_j$ is the same as 
the distance between $y_i$ and $y_j$. Therefore, we have
\begin{align*}
|x_i-x_j|=|y_i-y_j|=|\theta_i-\theta_j|\le \gamma.
\end{align*}
This implies that $\|B^{\top}\mathbf{x}\|_{\infty}\le
\gamma$. Therefore, $\theta\in S^{G}(\gamma)$. 

Regarding part~\ref{p2:exp}, using Lemma~\ref{lem:existence_path}, it is a
straightforward exercise to show that if
$\mathbf{x}^{\theta}=(x_1,\ldots,x_n)$, then we have
$|x_i-x_j|=|\theta_i-\theta_j|+2\pi k$, for some $k\in \mathbb{Z}_{\ge
  0}$. This implies that $\exp(\imagunit \mathbf{x}^{\theta})\in [\theta]$.

Regarding part~\ref{p3:embeded_cohesive}, suppose that $\mathbf{x}^{\theta}\in \bperp$. Then we have
$\|B^{\top}\mathbf{x}^{\theta}\|_{\infty}\le\gamma$. Since
$[\theta]=[\exp(\imagunit \mathbf{x}^{\theta})]$, it is clear that
$[\theta]\in S^{G}(\gamma)$.

Now suppose that $[\theta]\in S^{G}(\gamma)$, we will show that
$\mathbf{x}^{\theta}\in \bperp$. Assume that
$\mathbf{x}^{\theta}\not\in \bperp$. Therefore, there exists $(a,b)\in
\mathcal{E}$ such that $|x_a-x_b|>\gamma$. However, by the definition
of the set $S^{G}(\gamma)$, there exists $\mathbf{y}\in
\vect{1}_n^{\perp}$ such that $\theta=[\exp(\imagunit \mathbf{y})]$
and $\|B^{\top}\mathbf{y}\|_{\infty}\le\gamma$.  This implies that
$\theta=[\exp(\imagunit \mathbf{x})]=[\exp(\imagunit
  \mathbf{y})]$. Thus, there exists a vector
$\mathbf{v}\in\vect{1}_n^{\perp}$ whose components are integers such
that $\mathbf{y}=\mathbf{x}+2\pi\mathbf{v}$. Since we have
$|x_a-x_b|>\gamma$ and $|y_a-y_b|\le \gamma$, we cannot have $x_a=y_a$
and $x_b=y_b$. Therefore, $v_a\ne 0$ and we have $y_a=x_a+2\pi v_a$.
However, by Lemma~\ref{lem:existence_path}, since $[\theta]\in
S^G(\gamma)\subseteq \Delta^{G}(\gamma)$, there exists a simple path
$(i_1,\ldots ,i_r)$ between the nodes $a$ and $b$ such that, for every
$k\in \{1,\ldots,r\}$, we have $|x_{i_k}-x_{i_{k-1}}|\le\gamma$.
Since $|x_a-x_{i_2}|\le \gamma$ and $|y_{i}-y_{i_2}|\le \gamma$ we
have $y_{i_2}=x_{i_2}+2\pi v_a$. Similarly, one can show that
$y_{i_k}=x_{i_k}+2\pi v_a$, for every $k\in \{1,\ldots,r\}$.  This
implies that $y_b=x_b+2\pi v_a$. However, this means that
$|x_a-x_b|=|y_a-y_b|\le \gamma$, which is a contradiction. 

Regarding part~\ref{p4:diffeomorphic_S}, we first show that the set
$\bperp$ is compact. Since $\bperp\subset\real^{n}$, it suffice to
show that it is closed and bounded. We first show that, for every
$\mathbf{x}\in \bperp$, we have
\begin{align*}
-m\pi \le x_i\le m\pi,\qquad\forall i\in \until{n},
\end{align*}
where $m$ is the number of edges. Suppose that $\mathbf{x}\in \bperp$
and for some $k\in \until{n}$ we have $x_k>m\pi$. Since
$\mathbf{x}\in \bperp$, we get
\begin{align*}
|x_i-x_j|\le \gamma,\qquad\forall i,j\in\until{n}. 
\end{align*}
On the other hand, $\gamma\in [0,\pi)$ and $G$ is
connected. Therefore, for every $i\in \until{n}$, there exists a simple
path $(i_{1},i_2,\ldots,i_r)$ of length at most $m$ such that $i_1=1$
and $i_r=k$. This implies that 
\begin{align*}
  |x_k-x_i| &= |x_k-x_{i_{r-1}}+x_{i_{r-1}}-x_{i_{r-2}}+\ldots+x_{i_2}-x_1|\\ 
  &\le
  |x_k-x_{i_{r-1}}|+|x_{i_{r-1}}-x_{i_{r-2}}|+\ldots+|x_{i_2}-x_1|\\ &\le
  m\gamma
  <m\pi.
\end{align*}
Therefore, for every $i\in\until{n}$, we have $x_i>0$. As a
result, we get $\vect{1}_n^{\top}\mathbf{x}=\sum_{i=1}^{n}x_i>0$,
which is a contradiction since $\mathbf{x}\in \bperp$ and we
have $\mathbf{x}\in\vect{1}_n^{\perp}$. Similarly, one can show that,
for every $\mathbf{x}\in\bperp$, we have $-m\pi\le x_i$, for
every $i\in \until{n}$. Therefore, $\bperp$ is
bounded. The closedness of the set $\bperp$ is clear from
continuity of the $\infty$-norm. This implies that $\bperp$
is compact. Now we define the map $\xi: \bperp\to S^{G}(\gamma)$ by
$\xi(\mathbf{x}) = [\exp(\imagunit \mathbf{x})]$. We show that $\xi$ is a
real analytic diffeomorphism. It is easy to check that, for every $\mathbf{x}\in \bperp$, $D_{\mathbf{x}}\xi$ is
an isomorphism. Therefore $\xi$ is local diffeomorphim for every
$\mathbf{x}\in \bperp$. Now we show that $\xi$ is
one-to-one on the set $\bperp$. Suppose that, for $\mathbf{x},\mathbf{y}\in \bperp$, we
have $\exp(\imagunit\mathbf{x})=\exp(\imagunit\mathbf{y})$. Therefore,
we get $\mathbf{x}=\mathbf{y}+2\pi \mathbf{v}$, where $\mathbf{v}\in \vect{1}_n^{\perp}$ is a vector whose components
are integers. We will show that $\mathbf{v}=\vect{0}_n$. Suppose that
$\mathbf{v}\ne\vect{0}_n$. Since graph $G$ is connected, there
exists $(i,j)\in \mathcal{E}$ such that $v_i\ne v_j$. This implies that
\begin{align}\label{eq:inequality_angle}
x_i-x_j=y_i-y_j+2\pi(v_i-v_j).
\end{align}
Since we have $\|B^{\top}\mathbf{y}\|_{\infty}\le \gamma$, we get
$|y_i-y_j|\le \gamma$. However, by equation~\eqref{eq:inequality_angle}, we have
\begin{align*}
|x_i-x_j|&=|y_i-y_j+2\pi(v_i-v_j)|\ge 2\pi|v_i-v_j|-|y_i-y_j|\\ & \ge
  2\pi-\gamma > \pi.
\end{align*}
However, this is a contradiction with the fact that $\mathbf{x}\in
\bperp$. Therefore, $\mathbf{v}=\vect{0}_n$ and the map $\xi$ is
one-to-one. Note that by part~\ref{p3:embeded_cohesive}, $\xi$ is also
surjective. Therefore, using~\cite[Corollary 7.10]{JML:03}, the map $\xi$ is a
diffeomorphism between $\bperp$ and $S^{G}(\gamma)$. This completes the proof of part~\ref{p4:diffeomorphic_S}.
\end{proof}

\section{Proof of Theorem~\ref{thm:existence_uniqueness_S}}\label{app:real_analytic}

Regarding part~\ref{p4:equivalence}, suppose that $\mathbf{x}^*$ is a
synchronization manifold for the Kuramoto model~\eqref{eq:synchronization_manifold-x}. Then
$\omega=B\mathcal{A}\sin(B^{\top}\mathbf{x}^*)$.  By left-multiplying both
side of this equation by $B^{\top}L^{\dagger}$, we get
\begin{align*}
B^{\top}L^{\dagger}\omega=B^{\top}L^{\dagger}B\mathcal{A}\sin(B^{\top}\mathbf{x}^*)=\prj\sin(B^{\top}\mathbf{x}^*)=\fK(\mathbf{x}^*).
\end{align*} 
On the other hand, suppose that $\mathbf{x}^*$ satisfies the
edge balance equation~\eqref{eq:synchronization_manifold-edges}. Then, if we left-multiply both side of this equation by $B\mathcal{A}$, we get
\begin{align*}
B\mathcal{A}B^{\top}L^{\dagger}\omega=B\mathcal{A}\fK(\mathbf{x}^*)=B\mathcal{A}(B^{\top}L^{\dagger}B\mathcal{A})\sin(B^{\top}\mathbf{x}^*)
\end{align*} 
Noting that $B\mathcal{A}B^{\top}=L$ and $LL^{\dagger}=I_n-\frac{1}{n}\vect{1}_n\vect{1}_n^{\top}$, we have
\begin{align*}
(I_n-\frac{1}{n}\vect{1}_n\vect{1}_n^{\top})\omega=(I_n-\frac{1}{n}\vect{1}_n\vect{1}_n^{\top})B\mathcal{A}\sin(B^{\top}\mathbf{x}^*). 
\end{align*} 
Moreover, we have $\vect{1}_n^{\top}B=0$ and since $\omega\in
\vect{1}_n^{\perp}$, we get $\vect{1}_n^{\top}\omega=0$. This implies
that $\omega=B\mathcal{A}\sin(B^{\top}\mathbf{x}^*)$.  This completes
the proof of part~\ref{p4:equivalence}.

Regarding part~\ref{p3:f_one_to_one}, since the function $\sin$ is real analytic, $\fK$ is real analytic. The proof of injectivity of $\fK$ on
the domain $S^{G}(\gamma)$ is a straightforward generalization of~\cite[Corollary 2]{AA-SS-VP:81}. 

Regarding part~\ref{p5:existence_uniqueness_S}, by part~\ref{p3:f_one_to_one} the map $\fK$ is one-to-one on
the domain $S^{G}(\gamma)$. Therefore, if there exists a
synchronization manifold $\mathbf{x}^*\in S^{G}(\gamma)$, it is
unique. The proof of the fact that $\mathbf{x}^*$ is locally
exponentially stable is given in~\cite[Lemma 2]{FD-MC-FB:11v-pnas}.


\section{Proof of Lemma~\ref{thm:maf-non-zero} and of a useful equality}\label{app:maf}

\begin{proof}[Proof of Lemma~\ref{thm:maf-non-zero}]
By definition of the minimum amplification factor, it is clear that, for every
  $\gamma\in [0,\tfrac{\pi}{2})$ and every
  $p\in[1,\infty)\union\{\infty\}$, we have $\alpha_p(\gamma)\ge
  0$. So, to prove the lemma, it suffices to show that, for every $\gamma\in [0,\frac{\pi}{2})$ and every $p\in
  [1,\infty)\cup\{\infty\}$, we have $\alpha_p(\gamma)\ne 0$. Suppose that for some
  $\gamma\in [0,\frac{\pi}{2})$ and some $p\in
  [1,\infty)\cup\{\infty\}$, we have $\alpha_p(\gamma)=0$. Since
  $D_p(\gamma)$ and $\{\mathbf{z}\in \Img(B^{\top})\mid
  \|\mathbf{z}\|_{p}=1\}$ are compact sets, there exist $\mathbf{y}_0\in
  D_p(\gamma)$ and $\mathbf{z}_0\in \Img(B^{\top})$ with the property
  that $\|\mathbf{z}_0\|_p=1$ such that
 \begin{align*}
 \prj\diag(\sinc(\mathbf{y}_0))\mathbf{z}_0=\vect{0}_m.
 \end{align*}
   By premultipling both side of the above equality by
   $\mathbf{z}_0^{\top}\mathcal{A}$, we get
   $\mathbf{z}_0^{\top}\mathcal{A}\prj\diag(\sinc(\mathbf{y}_0))\mathbf{z}_0=0$.
   On the other hand, we know
 \begin{align*} 
 \mathbf{z}_0^{\top}\mathcal{A}\prj=\mathbf{z}_0^{\top}\mathcal{A}B^{\top}L^{\dagger}B\mathcal{A}=\mathbf{z}^{\top}_0\prj^{\top}
   \mathcal{A} = \mathbf{z}^{\top}_0 \mathcal{A},
 \end{align*}
where the last equality is a direct consequence of $\prj \mathbf{z}_0 = \mathbf{z}_0$. This implies that
 $\mathbf{z}_0^{\top}\mathcal{A}\diag(\sinc(\mathbf{y}_0))\mathbf{z}_0=0$.
 Since both $\mathcal{A}$ and $\diag(\sinc(\mathbf{y}_0))$ are
 diagonal positive definite matrices, we have
 $\mathbf{z}_0=\vect{0}_m$.  This is a contradiction with the fact
 that $\|\mathbf{z}_0\|_p=1$.
\end{proof}

The next result connects the minimum gain of an invertible
operator $T$ with the norm of $T^{-1}$. 

\begin{lemma}\label{thm:min-amplification}
  Let $(V,\|\cdot\|_V)$ be a normed real vector space and
  $\map{T}{V}{V}$ be a bijective linear map. Then
  \begin{align*}
    \min\setdef{\|T\mathbf{x}\|_V}{\|\mathbf{x}\|_V=1}=\frac{1}{\|T^{-1}\|_{V}}.
  \end{align*} 
\end{lemma}
\begin{proof}
  It is well-known and elementary to show that
  \begin{align}\label{eq:scaling}
    \min\setdef{\|T\mathbf{x}\|_V}{\|\mathbf{x}\|_V=1}=\min\left\{\frac{\|T\mathbf{x}\|_V}{\|\mathbf{x}\|_V}\
    \Big\vert \ \mathbf{x}\ne 0\right\}.
  \end{align}
  Because the linear map $T$ is invertible, for every $\mathbf{x}\in
  V$ such that $\mathbf{x}\ne \mathbf{0}$,
  \begin{align}\label{eq:recip}
    \frac{\|T\mathbf{x}\|_{V}}{\|\mathbf{x}\|_V}
    =\frac{1}{\frac{\|\mathbf{x}\|_V}{\|T\mathbf{x}\|_V}}=\frac{1}{\frac{\|T^{-1}\mathbf{y}\|_V}{\|\mathbf{y}\|_V}},
  \end{align}
  where $\mathbf{y}=T\mathbf{x}$.  Therefore, by taking the minimum of
  both side of the equation~\eqref{eq:recip} over $\mathbf{x}\ne
  \mathbf{0}$, we obtain
\begin{multline*}
  \min\left\{\frac{\|T\mathbf{x}\|_{V}}{\|\mathbf{x}\|_V}\ \Big\vert \
  \mathbf{x}\ne 0\right\}  =\min
  \left\{\frac{1}{\frac{\|T^{-1}\mathbf{y}\|_V}{\|\mathbf{y}\|_V}}\
  \Big\vert \ \mathbf{y}\ne \mathbf{0}\right\}\\  =\frac{1}{\max\left\{\frac{\|T^{-1}\mathbf{y}\|_V}{\|\mathbf{y}\|_V}\
  \Big\vert \ \mathbf{y}\ne \mathbf{0}\right\}} =\frac{1}{\|T^{-1}\|_V}.
\end{multline*}
 The result follows by recalling
equation~\eqref{eq:scaling}.
\end{proof}

\section{Proof of Lemma~\ref{lem:Q-inverse}}\label{app:lem:Q-inverse}

For every $\mathbf{y}\in D_{p}(\gamma)$, define the linear operator
$T_{\mathbf{y}}:\Img(B^{\top})\to \Img(B^{\top})$ by 
\begin{align*}
T_{\mathbf{y}}(\mathbf{z}) = B^{\top}L^{\dagger}_{\sinc(\mathbf{y})}B\mathcal{A}\mathbf{z}.
\end{align*} 
Let $\mathbf{z}\in \Img(B^{\top})$. Then there exists $\xi\in \vect{1}_n^{\perp}$ such that $\mathbf{z}=B^{\top}\xi$. This
implies that
\begin{align*}
Q(\mathbf{y})\scirc T_{\mathbf{y}} (\mathbf{z})&=Q(\mathbf{y})(B^{\top}L^{\dagger}_{\sinc(\mathbf{y})}B\mathcal{A})(B^{\top}\xi)\\
  & =\prj\diag(\sinc(\mathbf{y}))(B^{\top}L^{\dagger}_{\sinc(\mathbf{y})}L\xi)
  \\ & = B^{\top}L^{\dagger}L_{\sinc(\mathbf{y})}L^{\dagger}_{\sinc(\mathbf{y})}L\xi. 
\end{align*}
Since $\mathbf{y}\in D_p(\gamma)$, we have
$\dim(\Img(L_{\sinc(\mathbf{y})}))=n-1$. This implies that,
\begin{align*}
L^{\dagger}_{\sinc(\mathbf{y})}L_{\sinc(\mathbf{y})}=L_{\sinc(\mathbf{y})}L^{\dagger}_{\sinc(\mathbf{y})}=I_n-\tfrac{1}{n}\vect{1}_n\vect{1}_n^{\top}.
\end{align*}
Therefore, we get
\begin{align*}
Q(\mathbf{y})\scirc T_{\mathbf{y}} (\mathbf{z})=B^{\top}L^{\dagger}L\xi=B^{\top}\xi=\mathbf{z}.
\end{align*}
Since both $Q(\mathbf{y})$ and $T_{\mathbf{y}}$ are linear operators on
$\Img(B^{\top})$, we deduce that $Q(\mathbf{y})$ is
invertible and, for every $\mathbf{z}\in \Img(B^{\top})$:
\begin{align*}
(Q(\mathbf{y}))^{-1}\mathbf{z}=T_{\mathbf{y}}(\mathbf{z}) = B^{\top}L^{\dagger}_{\sinc(\mathbf{y})}B\mathcal{A}\mathbf{z}.
\end{align*}
This completes the proof of the lemma.

\section{A useful result and proof of Lemma~\ref{lem:inequality}}\label{app:inequality}
Regarding part~\ref{p1:inequality}, Note that $L$ is real and
  symmetric. Therefore, using singular-value
decomposition~\cite[Corollary 2.6.6]{RAH-CRJ:12}, there exists an orthogonal
matrix $U\in \real^{n\times n}$ such that:
\begin{align*}
  L=U\diag(0,\lambda_2,\ldots,\lambda_n)U^\top,
\end{align*}
where $0<\lambda_2\le \ldots \le \lambda_n$ are ordered eigenvalues of
$L$. Then, using~\cite[Chapter 6, \S 2, Corollary 1]{ABI-TNEG:03} the matrix $L^{\dagger}$ has the following
singular-value decomposition:
\begin{align*}
L^{\dagger}=U\diag\left(0,\frac{1}{\lambda_2},\ldots,\frac{1}{\lambda_n}\right)U^\top.
\end{align*}
Since $\omega\in \vect{1}_n^{\perp}=\Img(B)$, there exists
$\mathbf{y}\in \real^n$ which satisfies $\omega=B\mathbf{y}$. Therefore, we can write
\begin{align*}
  \left\| B^{\top}L^{\dagger}\omega \right\|_2^2
  &=\left\|B^{\top}U\diag\left(0,\frac{1}{\lambda_2},\ldots,\frac{1}{\lambda_n}\right)U^\top B\mathbf{y}\right\|_2^2
  \\
  &=\mathbf{y}^{\top}\left(B^{\top}U\diag\left(0,\frac{1}{\lambda_2},\ldots,\frac{1}{\lambda_n}\right)U^\top B\right)^2\mathbf{y}.
\end{align*}
Since we have
\begin{align*}
\diag\left(0,\frac{1}{\lambda_2},\ldots,\frac{1}{\lambda_n}\right)&\preceq
\diag\left(\frac{1}{\lambda_2},\frac{1}{\lambda_2},\ldots,\frac{1}{\lambda_2}\right),\\
\mathbf{0}&\preceq
B^{\top}U\diag\left(0,\frac{1}{\lambda_2},\ldots,\frac{1}{\lambda_n}\right)U^\top
  B,
\end{align*}
we obtain
\begin{align*}
   \left(B^{\top}U\diag\left(0,\frac{1}{\lambda_2},\ldots,\frac{1}{\lambda_n}\right)U^\top B\right)^2
  &\preceq\frac{1}{\lambda^2_2}\left(B^{\top}UU^\top B\right)^2.
\end{align*}
Therefore, 
\begin{align}
  \left\| B^{\top} L^{\dagger} B\mathbf{y}\right\|^2_2
  &\le \frac{1}{\lambda_2^2}\mathbf{y}^{\top}\left(B^{\top}UU^\top B\right)^2\mathbf{y}
  \label{eq:ineq:tmp}
  \\
  &=\frac{1}{\lambda_2^2}\left\|B^{\top}UU^\top B\mathbf{y}\right\|^2_{2}
  =\frac{1}{\lambda_2^2}\left\|B^{\top}\omega\right\|^2_2.
  \nonumber
\end{align}
This concludes the proof of inequality. Regarding the equality,
suppose that $i\in \{2,\ldots,n\}$ is the smallest positive integer such that
$\lambda_i\ne \lambda_2$. Note that, by the above analysis, the
equality holds for $\omega=B\mathbf{y}\in \vect{1}_n^{\perp}$ if and only if 
\begin{align}\label{eq:equality1}
\left\| B^{\top}
  L^{\dagger}B\mathbf{y}\right\|^2_2=\frac{1}{\lambda_2^2}\mathbf{y}^{\top}\left(B^{\top}UU^\top
  B\right)^2\mathbf{y}. 
\end{align}
We define the diagonal matrix $\Lambda$ by
\begin{align*}
  \Lambda=\diag\Big(\tfrac{1}{\lambda_2},\
  \underbrace{0,\ldots,0}_{i-1}\ ,\left(\tfrac{1}{\lambda_2}-\tfrac{1}{\lambda_i}\right),\ldots,
  \left(\tfrac{1}{\lambda_2}-\tfrac{1}{\lambda_n}\right) \Big).
\end{align*}
This implies that the equality~\eqref{eq:equality1} holds if and only if 
\begin{align}\label{eq:equality2}
\mathbf{y}^{\top}B^{\top}U\Lambda U^{\top}B^{\top}\mathbf{y}=0. 
\end{align}
Since $U\Lambda U^{\top}\succeq 0$, the equality~\eqref{eq:equality2}
holds if and only if $\omega=B^{\top}\mathbf{y}\in \Ker(U\Lambda
U^{\top})$. However, we know that
\begin{align*}
  \Ker(U\Lambda U^{\top})= \mathrm{span}\{v_2,\ldots,v_{i-1}\},
\end{align*}
where $v_k$ is the eigenvector associated to the eigenvalue
$\lambda_k$. This completes the proof of part~\ref{p1:inequality}. Regrading part~\ref{p2:T0_T2}, if $\omega\in \vect{1}^{\perp}_n$
satisfies test~\eqref{test:best-known}, then we have
$\|B^{\top}\omega\|_{2}<\lambda_2(L)$. Therefore, by part~\ref{p1:inequality},
we have 
\begin{align*}
\|B^{\top}L^{\dagger}\omega\|_{2} \le \frac{1}{\lambda_2(L)}
  \|B^{\top}\omega\|_2<1.
\end{align*}
This means that $\omega$ satisfies test~\eqref{test:2-norm-new}.

\bibliographystyle{plainurl}
\bibliography{alias,Main,FB}

%

\begin{IEEEbiography}[{\includegraphics[width=1in,height=1.25in,clip,keepaspectratio]{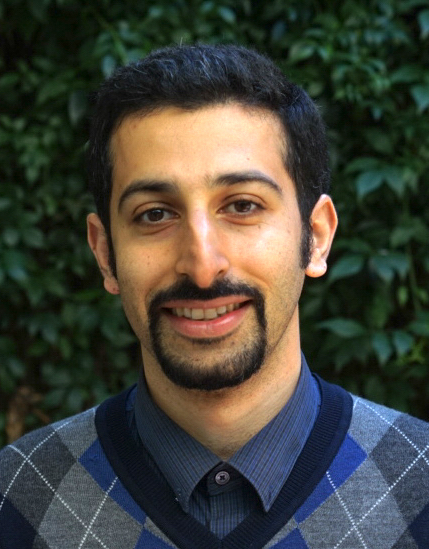}}]{Saber
   Jafarpour}(M'16) is a Postdoctoral researcher with the Mechanical
  Engineering Department and the Center for Control, Dynamical Systems
  and Computation at the University of California, Santa Barbara. He
  received his Ph.D. in 2016 from the Department of Mathematics and
  Statistics at Queen's University. His research interests include
  analysis of network systems with application to power grids and
  geometric control theory.
\end{IEEEbiography}

\begin{IEEEbiography}[{\includegraphics[width=1in,height=1.25in,clip,keepaspectratio]{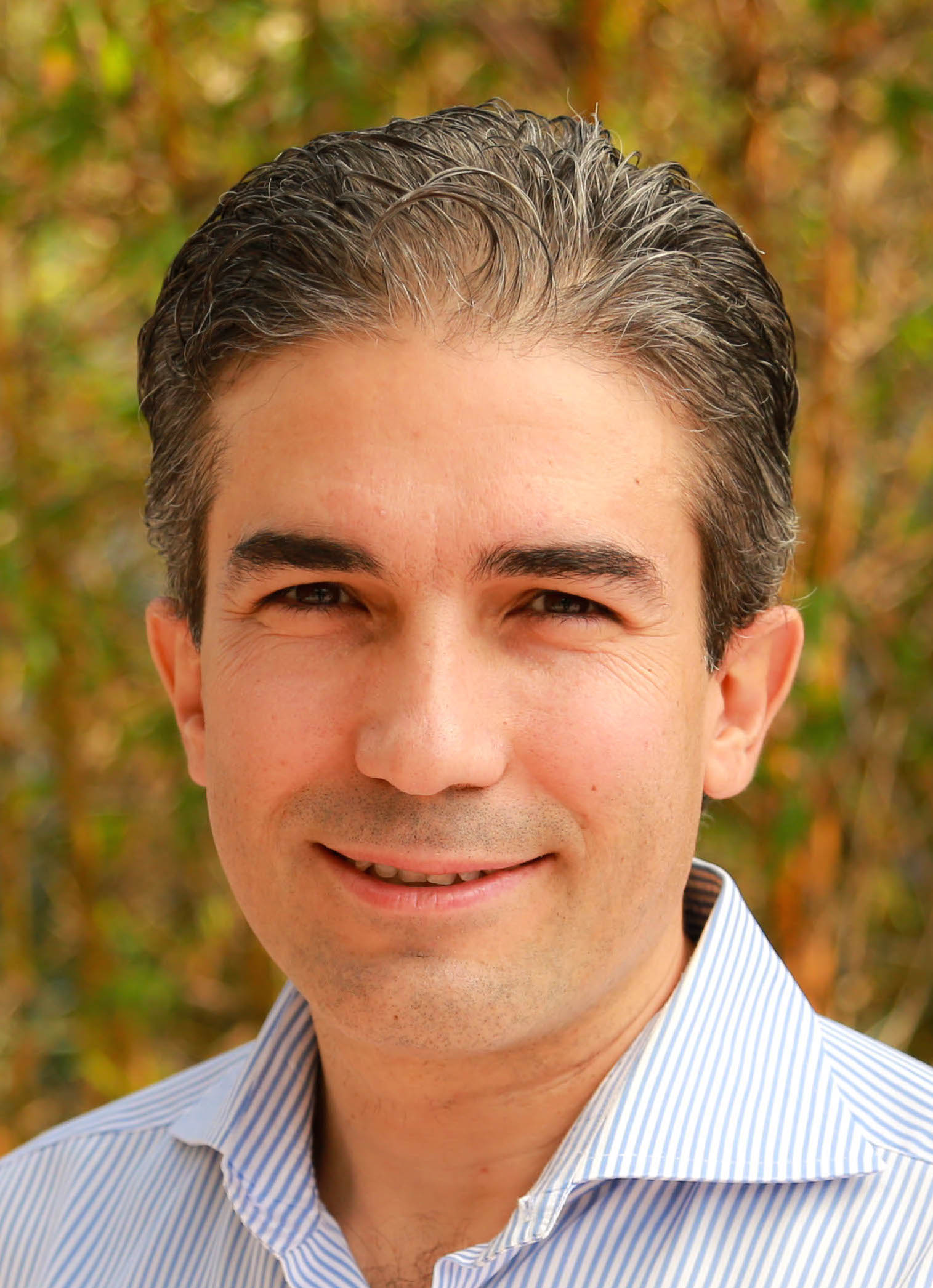}}]{Francesco Bullo} (IEEE S'95-M'99-SM'03-F'10) is a Professor with the Mechanical Engineering Department and the
 Center for Control, Dynamical Systems and Computation at the University of
 California, Santa Barbara. He was previously associated with the
 University of Padova, the California Institute of Technology, and the
 University of Illinois. His research interests focus on network systems
 and distributed control with application to robotic coordination, power
 grids and social networks. He is the coauthor of “Geometric Control of
 Mechanical Systems” (Springer, 2004) and “Distributed Control of Robotic
 Networks” (Princeton, 2009); his forthcoming "Lectures on Network Systems"
 is available on his website.  He received best paper awards for his work
 in IEEE Control Systems, Automatica, SIAM Journal on Control and
 Optimization, IEEE Transactions on Circuits and Systems, and IEEE
 Transactions on Control of Network Systems.  He is a Fellow of IEEE and
 IFAC.  He has served on the editorial boards of IEEE, SIAM, and ESAIM
 journals, and serves as 2018 IEEE CSS President.
\end{IEEEbiography}





\end{document}